\documentclass[12pt, reqno]{amsart}
\usepackage{amsmath, amsthm, amscd, amsfonts, amssymb, graphicx, color, mathrsfs}
\usepackage[bookmarksnumbered, colorlinks, plainpages]{hyperref}
\usepackage[all]{xy} 
\usepackage{slashed}

\newtheorem{theorem}{Theorem}[section]
\newtheorem{lemma}[theorem]{Lemma}

\theoremstyle{definition}

\theoremstyle{remark}
\newtheorem{remark}[theorem]{Remark}
\numberwithin{equation}{section}

\begin{document}
\setcounter{page}{1}

\title[Finite rank perturbations]{Finite rank perturbation of non-Hermitian random matrices: heavy tail and sparse regimes}

\author[Y.H]{Yi Han}
\address{
  Yi Han:
  \endgraf
  Department of Mathematics, Massachusetts Institute of Technology, Cambridge MA
  \endgraf
  {\it E-mail address} {\rm hanyi16@mit.edu}
  }

\thanks{The author was supported by an EPSRC grant EP/W524141/1 when part of the work was carried out. The author was then supported by a Simons Foundation Grant (601948, DJ) when the rest of the work was completed.
}

\begin{abstract} In this work we investigate spectral properties of squared random matrices with independent entries that have only two finite moments. We revisit the problem of perturbing a large, i.i.d. random matrix by a finite rank error. We prove that under a merely second moment condition, for a large class of perturbation matrix with bounded rank and bounded operator norm, the outlier eigenvalues of perturbed matrix still converge to that of the perturbation, which was previously known when matrix entries have finite fourth moment. We then show that the same perturbation holds for very sparse random matrices with i.i.d. entries, all the way up to a constant number of nonzero entries per row and column.
\end{abstract} 

\maketitle

\allowdisplaybreaks

\section{Introduction}

Consider $A_n:=\{a_{ij}\}_{1\leq i,j\leq n}$ an $n$ by $n$ random matrix with i.i.d. elements $a_{ij}$ having mean zero and finite second moment: \begin{equation}\label{firstcondition}\mathbb{E}[a_{ij}]=0 \text{ , and } \mathbb{E}[|a_{ij}|^2]=1.\end{equation} When $a_{ij}$'s have complex Gaussian distribution, then $A_n$ belongs to the complex Ginibre ensemble \cite{ginibre1965statistical}. For $A_n$ with i.i.d. entries with mean $0$, variance one and otherwise arbitrary distribution, it is known (\cite{ WOS:000281425000010}, culminating many previous works, see references therein) that the empirical spectral density of $n^{-1/2}A_n$ converges to the circular law, i.e. the uniform distribution in the unit disk in the complex plane.

Determining the spectral radius $\rho(A_n)$, and further determining the outliers of $A_n$ under a bounded rank noise, is a very different task.
 When entries $a_{ij}$ have finite fourth moments, applying a high trace method to $n^{-1/2}A_n$ and using a standard truncation procedure (see \cite{Bai1986LimitingBO} \cite{Geman1986THESR}) proves that $\rho(A_n)$
converges to 1 almost surely. When $a_{ij}$ has a finite second moment but infinite fourth moment, then the operator norm of $n^{-1/2}A_n$ tends to infinity, but it was later uncovered that under merely assumption \ref{firstcondition}, $\rho(n^{-1/2}A_n)$ still converges to 1 as $n$ tends to infinity, i.e. no outlier exists for the circular law. This no-outlier result was first proven via a high trace moment method in \cite{ WOS:000435416700013} under additional conditions (assuming finite $2+\eta$-moments and a symmetric entry distribution), and then proven in full generality via a characteristic functions method proposed in \cite{bordenave2021convergence}. 

Another model of interest concerns adjacency matrices of sparse Erdős–Rényi digraphs, that is, we consider $A_n=\{a_{ij}\}_{1\leq i,j\leq n}$ with each $a_{ij}$ i.i.d. having distribution $\operatorname{Ber}(\frac{d_n}{n}),d_n>1$. For these sparse matrices, we know that (\cite{benaych2020spectral},\cite{coste2023sparse}) there is only one outlier eigenvalue of $(d_n)^{-1/2}A_n$ (tending to infinity when $d_n\to\infty$ or converging to $\sqrt{d}$ for $d_n=d$), and all other eigenvalues have modulus $1+o(1)$. The circular law for $(d_n)^{-1/2}A_n$ in the regime $d_n\to\infty$ was recently proved in full generality in \cite{rudelson2019sparse}, and the existence of limiting spectral density in the case of constant $d$ was recently resolved in \cite{sah2023limiting}.

A closely related stream of research concerns finite rank perturbations of random matrices, that is, what are the outlier eigenvalues of $n^{-1/2}A_n+C_n$, where $C_n$ is a matrix with finite rank and bounded operator norm? The corresponding problem for Hermitian random matrices has been studied in more detail in \cite{peche2006largest}, \cite{capitaine2012central}, \cite{pizzo2013finite}, \cite{feral2007largest}, \cite{benaych2011eigenvalues} (see also references therein): there is a phase transition called the BBP transform: for a Wigner matrix $W_n$, a unit vector $v$ and $\theta>0$, the largest eigenvalues of $n^{-1/2}W_n+\theta vv^t$ converges to $\theta+\frac{1}{\theta}$ when $\theta>1$ and converges to 2 when $\theta<1$. In the non-Hermitian case of $A_n$ with i.i.d. elements, the outlier eigenvalues of $n^{-1/2}A_n+C_n$ behave rather differently: 
it is first proved in Tao \cite{tao2013outliers} that in this i.i.d. case, under a fourth moment assumption on matrix entries, the outlier eigenvalues of $n^{-1/2}A_n+C_n$ converge towards the eigenvalues of $C_n$ with modulus larger than one, so that the mapping $\theta\to \theta+\frac{1}{\theta}$ in the Wigner case is replaced by the identity mapping in this i.i.d. case when $|\theta|>1$.  Fluctuations around the limiting value were then studied in \cite{bordenave2016outlier}, and see also \cite{benaych2016outliers} for outliers in a related case: the single ring theorem. Meanwhile, finite rank perturbations for elliptic random matrices were studied in \cite{article1221} also under a finite fourth moment condition, and the result interpolates between the i.i.d. case and the Wigner case, giving a mapping $\theta\to \theta+\frac{\rho}{\theta}$ for $|\theta|>1$, where $\rho\in(-1,1)$ is the covariance profile of the matrix, i.e. $\mathbb{E}[a_{ij}a_{ji}]=\rho$.

The first objective of this paper is to show that the celebrated result $\rho(n^{-1/2}A_n)\to 1$ under a finite second moment condition in \cite{bordenave2021convergence} holds more generally for product matrices, that is, we have $\rho(\prod_{i=1}^m(n^{-1/2}A^i))\to 1$ for any fixed $m$ under finite second moment condition, where $A^1,\cdots,A^m$ are i.i.d. copies of $A_n$.

The second objective of this paper is to investigate the effect of infinite fourth moment and sparsity on the outlier eigenvalues of $n^{-1/2}A_n+C_n$. Technically, a lot of complications arise in this setting: the operator norm of $n^{-1/2}A_n$ is unbounded, so that arguments based on operator norms, as in \cite{bordenave2016outlier}, do not apply. We also cannot use the replacement principle to compare $A_n$ to a matrix with Gaussian entries as in \cite{tao2013outliers} as the replacement principle typically relies on a fourth moment assumption. Likewise,we find it difficult in the computations to use singular value decomposition and simplify $C_n$ into a canonical form, because the operator norm of $n^{-1/2}A_n$ is not bounded. We have to turn back to first principles. Inspired by \cite{bordenave2021convergence}, we use the characteristic function method to find outlier eigenvalues. 

The main results of this paper are listed as follows. We use the notation $H(\mathbb{D})$ to denote the space of holomorphic functions on the unit disk $\mathbb{D}$, equipped with the compact open topology, i.e. the topology of uniform convergence on compact subsets. A sequence $c_n\in H(\mathbb{D})$ is said to form a tight sequence if for any $\epsilon>0$ there is a compact subset $K_\epsilon\subset H(\mathbb{D})$ such that $\mathbb{P}(c_n\in K_\epsilon)\geq 1-\epsilon$ for each $n$.

\subsection{Spectral radius of product heavy-tailed i.i.d. matrices}

The first result of this paper proves that the spectral radius of a product of a fixed number of i.i.d. random matrices with only second moment still converges to 1 in probability. This generalizes the main result of \cite{bordenave2021convergence} to product matrices.

\begin{theorem}\label{spectralradiusproduct}
 Fix an integer $m\in\mathbb{N}_+$. Consider $A^1,\cdots,A^m$ $m$ independent $n\times n$ random matrices with entry $(A^k)_{ij}=a^k_{ij}$, where $(a^k_{i,j})_{1\leq i,j\leq n,k=1,\cdots,m}$ are independent identically distributed random variables satisfying 
    \begin{equation}
        \mathbb{E}[a_{ij}^k]=0,\text{and } 
  \mathbb{E}[|a_{ij}^k|^2]=1.   \end{equation}
  Then for any $\epsilon>0$, with probability $1-o(1)$, all eigenvalues of the product matrix $$n^{-m/2}A^1\cdots A^m$$ lie in $(1+\epsilon)\mathbb{D}$, i.e. 
  $$
\rho(n^{-m/2}A^1\cdots A^m)\leq 1+\epsilon,
  $$ where $\rho(A)$ denotes the spectral radius of a square matrix $A$.
\end{theorem}

The product of i.i.d. matrices has been studied by several authors before. For convergence of the empirical spectral distribution, \cite{MR2861673} studied the case when entries have finite $(2+\eta)$-moment and \cite{o2015products} studied the elliptic case. It is believed that in these works the $(2+\eta)$-moment condition can be weakened to only a second moment condition. For the spectral radius and perturbations under a finite rank noise, \cite{MR4076784} provided a very thorough study assuming that entries have a finite fourth moment. In this work we complement the picture by showing that the spectral radius for product i.i.d. matrices (with only unit second moment) also converges to one. 

The proof of Theorem \ref{spectralradiusproduct} relies on a linearization procedure to transfer the problem to a block matrix $\mathcal{Y}$ with independent entries. We then analyze the matrix $\mathcal{Y}$ by showing that its (rescaled) characteristic function converges to some Gaussian analytic function, which is an adaptation of the technique in \cite{bordenave2021convergence} to block type matrices.

After the paper was submitted, in a slightly more recent paper \cite{han2025spectral}, Theorem 1.11, the author used a different method to prove a generalized version of Theorem \ref{spectralradiusproduct} but assuming a stronger condition, namely $2+\epsilon$-finite moments. Another recent paper \cite{hachem2025spectral} also discusses the same problem. We still keep the proof of Theorem \ref{spectralradiusproduct} here to illustrate the proof technique, and we will later explore the outliers for the product of i.i.d. matrices in Theorem \ref{productoftheorem1}, where this computation in Theorem \ref{spectralradiusproduct} will be used.

A very interesting problem is to consider product elliptic random matrices (including product Wigner matrices) with only finite second moment assumption. The techniques of this paper (and also \cite{bordenave2021convergence}) are not directly applicable: a key technical step is to show the characteristic function forms a tight sequence, whose proof (see Lemma \ref{tightnesslemma}) generally requires all entries of $\mathcal{Y}$ are independent. The independence is guaranteed in the i.i.d. case we consider here, but is not granted in the elliptic case. When entries have finite $4+\epsilon$ moments, the spectral outliers can be determined by much more classical methods from free probability, see the companion paper \cite{han2024outliers} and many references therein.

\subsection{Finite rank perturbation: finite second moment}
The second series of results in this paper is to prove finite rank perturbation results for random matrices with only two finite moment. They are given as follows: 

\begin{theorem}\label{theorem1} Assume that $C_n=(C_{ij})_{1\leq i,j\leq n}$ is a deterministic $n\times n$ square matrix that belongs to one of the following two classes:
\begin{enumerate}
    \item $C_n$ is a diagonal matrix of finite rank, or $C_n$ is in a Jordan normal form with finite rank, or more generally $C_n$ has only $M_1$ non-zero entries. In all these cases, assume moreover that all the elements of $C_n$ are bounded in norm by $M_2$, and constants $M_1,M_2$ are independent of $n$,
    \item $C_{ij}=\frac{\mu}{n}$ for each $i,j$, or more generally $|C_{ij}|\leq \frac{M}{n}$ for some $M>0$ and each $i,j$, where $M$ is independent of $n$.
\end{enumerate}
Assume also that the following condition holds in each previous case: denote by $c_n(z)$ the reverse characteristic function of $C_n$: $c_n(z)=\det(\operatorname{I}-zC_n),$ then $c_n(z)$ forms a tight family in $\operatorname{H}(\mathbb{D})$, and there exists some non-zero function $c(z)\in \operatorname{H}(\mathbb{D})$ such that $c_n(z)$ converges to $c(z)$ in $\operatorname{H}(\mathbb{D})$ as $n$ tends to infinity.

Assume also that $A_n=(a_{ij})_{1\leq i,j\leq n}$ is a square matrix of size $n$, with independent, identically distributed entries $a_{ij}$ satisfying \begin{equation}\label{firstcondition12}\mathbb{E}[a_{ij}]=0 \text{ ,and } \mathbb{E}[|a_{ij}|^2]=1.\end{equation}
(In particular, $a_{ij}$ may not have a finite fourth moment.)

    Then almost surely, as $n$ tends to infinity, the eigenvalues of $n^{-1/2}A_n+C_n$ with modulus greater than $1+o(1)$ converge (after relabeling) to the solutions $z$ of $c(z^{-1})=0$ with $|z|>1$, and every solution to $c(z^{-1})=0$ with $|z|>1$ arises as the limit of some eigenvalues of $n^{-1/2}A_n+C_n$.

    Along the proof we will also prove the following convergence of characteristic functions: denote by $q_n(z)=\det(\operatorname{I}-z(n^{-\frac{1}{2}}A_n+C_n))$, then $q_n(z)$ converges in distribution to $q(z)$ as functions in $\operatorname{H}(\mathbb{D})$, where 
    \begin{equation}
        q(z)=\kappa(z)c(z)e^{-F(z)},\quad z\in\mathbb{D},
    \end{equation} 
    where $\kappa(z)=\sqrt{1-\mathbb{E}[a_{11}^2]z^2}$, $F(z)=\sum_{k=1}^\infty X_k\frac{z^k}{k} 
$ where $\{X_k\}_{k\geq 1}$ are independent complex Gaussian variables with 
$$
\mathbb{E}[X_k]=0,\quad \mathbb{E}[|X_k|^2]=1,\quad \mathbb{E}[X_k^2]=\mathbb{E}[a_{11}^2]^k.
$$
\end{theorem}

As a special case, when the finitely many eigenvalues of $C_n$ are $n$-independent, then our assumption that $c_n(z)$ converges to $c(z)$ is always guaranteed.

The proof of both Theorem \ref{spectralradiusproduct} and \ref{theorem1} are based on a trace moment method proposed in \cite{bordenave2021convergence} that yields the convergence of spectral radius via characteristic polynomials. An interesting technical step in the proof, which we would like to highlight in the introduction, is the following. Let $A_n$ be a random matrix we care about (which may not be the same $A_n$ in Theorem \ref{spectralradiusproduct} and \ref{theorem1})). We define $q_n(z)=\det(1-n^{-1/2}zA_n)$ the reverse characteristic polynomial. We can expand $q_n(z)=1+\sum_{k=1}^\infty(-z)^kP_k^n$ for some random coefficients $P_k^n$. At least when $z$ is sufficiently small, we have that the following defined power sequence 
$$B_n:=-\sum_{k=1}^\infty (\frac{A_n}{\sqrt{n}})^k\frac{z^k}{k}$$
is well-defined and converges to the logarithm of $1-z\frac{A_n}{\sqrt{n}}.$ From the matrix relation $\det(e^{B_n})=e^{\operatorname{Tr}B_n}$, this leads to, for $z$ sufficiently small, 
$$
q_n(z)=\exp\left(-\sum_{k=1}^\infty \frac{\operatorname{Tr}((A_n)^k)}{n^{k/2}}\frac{z^k}{k}
\right).
$$ Since two analytic functions that are equal in a small neighborhood of complex plane must be equal identically, this leads to the following consequence:

For each $k\in\mathbb{N}_+$, the coefficients $(P_1^n,\cdots,P_k^n)$ can be represented as ($n$-independent) polynomial functions of $$\left(\frac{\operatorname{Tr}(A_n)}{\sqrt{n}},\cdots,\frac{\operatorname{Tr}((A_n)^k)}{n^{k/2}}\right)$$ (and in reverse). Thanks to this expression, all the further effort of proving convergence of $q_n(z)$ boils down to proving the convergence of $\operatorname{Tr}(A_n)^k/n^{k/2}$ as $n$ tends to infinity. Our proof in each specific case is based on an explicit computation for the distributional limit of $\operatorname{Tr}((A_n)^k)$ where we replace the i.i.d. matrix $A_n$ by its perturbed version $A_n+\sqrt{n}C_n$.

The main technique of this paper is to carry out a direct computation of the characteristic function $\det(1-z(n^{-1/2}A_n+C_n))$. This is in contrast to the more standard approach in determining the outliers of perturbed random matrices, as proposed in \cite{tao2013outliers}, where we factorize the characteristic polynomial as 
$$
\det(z-(n^{-1/2}A_n+C_n))=\det(z-n^{-1/2}A_N)\det(I-(z-n^{-1/2}A_n)^{-1}C_n).
$$ Then we find a spectral decomposition of $C_n=U_n\Lambda_n V_n$ to reduce the second determinant to a finite dimensional determinant 
$$
det(I-(z-n^{-1/2}A_n)^{-1}C_n)=\det(I-\Lambda_nV_n(z-n^{-1/2}A_n)^{-1}U_n).
$$ While being very convenient to use, this standard method does not seem to work well in this paper as we do not know how to compute the resolvent $(z-n^{-1/2}A_n)^{-1}$ in the large $n$ limit. When the entries of $A_n$ have only two finite moments, it is not even known if the resolvent will converge to a deterministic limit. On the other hand, a direct computation of the characteristic function $\det(z-(n^{-1/2}A_n+C_n))$ can better capture the cancellation of the i.i.d. entries, but the proof becomes more involved and technically relies on the specific form of the perturbation $C_n$. The spectral decomposition $C_n=U_n\Lambda_n V_n$ cannot be used when we implement this method of direct computation.

Our assumptions on $C_n$ are a bit technical: in practice, one should be able to check that the result of Theorem \ref{theorem1} is true for any $C_n$ with bounded rank and bounded operator norm. The combinatorics involved in the computation for general $C_n$ is however very complicated, as we cannot use directly the singular value decomposition for $C_n$ (if $n^{-\frac{1}{2}}A_n$ has bounded operator norm, then SVD can be effectively used as we can ignore the unitary matrices in SVD and reduce $C_n$ to a normal form, but this is not the case here with only two finite moments). Under a finite fourth moment condition, this theorem (in more general form) was proved in \cite{tao2013outliers}, see also \cite{bordenave2016outlier}.

We also note that the constants $M_1$,$M_2$ and $M$ in the statement of Theorem \ref{theorem1} can be slowly growing in $n$ at rate $O(n^{o(1)})$ with exactly the same proof.

The main achievement of this result compared to previous works is that we do not assume entries $a_{ij}$ have a finite fourth moment, and under merely finite second moment, we have the same finite-rank perturbation result. Note that our proof does not yield the fluctuations of outlier eigenvalues around their limiting value, which is studied in \cite{bordenave2016outlier} for some special cases of $C_n$ under finite fourth moment assumption. Also, for (symmetric) Wigner matrices with entries having less than four moments, its finite rank perturbation result is recently obtained in \cite{diaconu2022finite} and is completely different from the non-Hermitian case here.

For symmetric (or Hermitian) random matrices, Van Handel and Brailovskaya \cite{brailovskaya2022universality} recently proposed a general framework, further developing the ideas in \cite{bandeira2023matrix}, that can prove finite rank perturbation results for random matrices only assuming very few moments. The main motivating example in \cite{brailovskaya2022universality} can be stated as follows: consider $W_n=(W_{ij})$ a real symmetric random matrix with $(W_{ij})_{1\leq i\leq j\leq n}$ independent, that $W_n$ has mean $0$ and has the same variance profile as a GOE matrix. Then assuming moreover that $n^{-\frac{1}{2}}|W_{ij}|\leq (\log n)^{-2}$ a.s. for each $i,j$, the finite rank perturbation properties of $W_n$ are asymptotically the same as a GOE matrix: we still have the phase change at $\theta=1$ and the mapping $\theta\to\theta+\frac{1}{\theta}$ when $|\theta|>1$, (the fluctuation around the spiked eigenvalues are different though). This result is remarkable in that only a second moment condition is assumed, but the condition $n^{-\frac{1}{2}}|W_{ij}|\leq (\log n)^{-2}$ is necessary: without it, the top eigenvalues are Poisson distributed, see \cite{diaconu2022finite}. The main result of this paper also assumes only a finite second moment, yet we do not impose any truncation condition and still get the same finite rank perturbation as the subGaussian case: this is because the i.i.d. nature of our matrix leads to more cancellations than Wigner matrices. For this reason, we reveal a picture very different from that in \cite{brailovskaya2022universality}. Also, \cite{brailovskaya2022universality} uses universality of resolvent (or Stieltjes transform) which is very convenient when eigenvalues are on the real line, but much more work is needed when eigenvalues are dispersed on the complex plane. 

In a companion paper \cite{han2024outliers}, the author further uses free probability ideas from \cite{bandeira2023matrix} to study finite rank perturbations for non-Hermitian random matrices. The results presented in \cite{han2024outliers} have a different flavor from those presented here: in \cite{han2024outliers}, we consider random band matrices that either have independent entries or have an elliptic covariance profile, whereas band matrices are outside the scope of techniques of this paper. On the other hand, the techniques in \cite{han2024outliers} use a Hermitization procedure, and therefore excludes the case with only two bounded moments and very sparse random graphs, as considered here. It seems very hard to find a theory that unifies both the method of characteristic functions as in this paper, and the method of free probability and concentration inequalities in \cite{han2024outliers}. We hope to find a general powerful method that unites both results in \cite{han2024outliers} and in this paper.

\begin{remark} Another important case consists of i.i.d matrix with non-zero mean, that is, we consider $C_{ij}=\frac{\mu}{\sqrt{n}}$. This case is however not covered by the techniques of this paper, as in this case the characteristic function of $n^{-1/2}A_n+C_n$ is no longer a perturbation of that of $n^{-1/2}A_n$, and it is not clear if the method of characteristic functions still applies. When $a_{ij}$ have finite fourth moment, it is proved in \cite{tao2013outliers} that except the single outlier eigenvalue at $\mu\sqrt{n}+o(1)$, all other eigenvalues lie in $\{z\in\mathbb{C}:|z|\leq 1+o(1)\}$. The author is not sure if exactly the same result can be proved under only a second moment condition, and this is left for future work. The existence of exceptional eigenvalue to this matrix model with non-zero mean was first noticed in \cite{andrew1990eigenvalues}, and more precise information on the distribution of this eigenvalue can be found in \cite{silverstein1994spectral} under additional moment conditions.
\end{remark}

\subsection{Finite rank perturbation: sparse i.i.d. matrices}
The third major result of this paper concerns sparse non-Hermitian random matrices. In the companion paper \cite{han2024outliers} the author considered a large family of sparse matrices which, with high probability, has at least, say, $(\log n)^5$ nonzero entries per row and column. This threshold is critical for the free probability approach in \cite{han2024outliers}.
In contrast, we show here that the perturbation result continues to hold for sparse matrices with constant, or very slowly growing, number of nonzero elements per row and column.

\begin{theorem}\label{theorem1.2}
\begin{enumerate}
\item Let $A_n$ be a $n\times n$ random matrix with i.i.d. Bernoulli $(\frac{d}{n})$ entries ($d>1$). Let $C_n=(C_{ij})_{1\leq i,j\leq n}$ be a deterministic $n\times n$ matrix with each entry bounded by $M_1$ and at most $M_2$ non-zero elements ($M_1,M_2$ are independent of $n$). Denote by $c_n(z)=\det(\operatorname{I}-zC_n)$ the reverse characteristic function and assume that $c_n(z)$ converges to $c(z)$ as holomorphic functions on $\{z\in\mathbb{C}:|z|<\frac{1}{\sqrt{d}}\}$. 

Then almost surely, as $n$ tends to infinity, all eigenvalues of $A_n+C_n$ lie in $\{z\in\mathbb{C}:|z|\leq \sqrt{d}(1+o(1))\}$ except the following outliers: an outlier eigenvalue converges to $d$ with multiplicity one, and all the other outliers converge to the solutions of $c(z^{-1})=0$ with $|z|>\sqrt{d}$. Meanwhile, each solution to $c(z^{-1})=0$ with $|z|>\sqrt{d}$ correspond to the limit of a sequence of outlier eigenvalues of $A_n+C_n$.

\item Let $A_n$ be a $n\times n$ random matrix with i.i.d. Bernoulli $(\frac{d_n}{n})$ entries where $d_n\to\infty$ as $n$ tends to infinity and $d_n=n^{o(1)}$. Let $C_n=(C_{ij})_{1\leq i,j\leq n}$ be a deterministic $n\times n$ matrix with each entry bounded by $M_1$ and at most $M_2$ non-zero elements ($M_1$ and $M_2$ are independent of $n$). Denote by $c_n(z)=\det(\operatorname{I}-zC_n)$ and assume that $c_n(z)$ converges to $c(z)$ as holomorphic functions on $\mathbb{D}$. 

Then almost surely, as $n$ tends to infinity, all eigenvalues of $(d_n)^{-\frac{1}{2}}A_n+C_n$ lie in $\{z\in\mathbb{C}:|z|\leq 1+o(1)\}$ except the following outliers: one outlier eigenvalue tends almost surely to infinity, and all the other outlier eigenvalues converge to a solution to $c(z^{-1})=0$ with $|z|>1$. Meanwhile, each solution to $c(z^{-1})=0$ with $|z|>1$ correspond to the limit of a sequence of outlier eigenvalues of $(d_n)^{-\frac{1}{2}}A_n+C_n$.
\end{enumerate}

\end{theorem}

We have also proved that the reverse characteristic function $\det(\operatorname{I_n}-z((d_n)^{-\frac{1}{2}}A_n+C_n))$ converges to some limiting random analytic functions, and the limit is closely related to random analytic functions derived in \cite{coste2023sparse}.

\begin{remark}
    Under the assumptions in case (2), it is proved in \cite{basak2019circular}, \cite{rudelson2019sparse} that the empirical spectral density of $A_n$ converges to the circular law. Very recently, existence of limiting empirical spectral density of $A_n$ under the assumption of case (1) was proved in \cite{sah2023limiting}.
\end{remark}

\begin{remark}
We have ruled out the case where $C_n$ has many non-zero elements, say for example $C_{i,j}=\frac{\mu}{n}$ for some $\mu\neq 0$.  In this case the rank one perturbation result should be different from the stated result: we expect that in case (1), the outlying eigenvalue may not converge to $d$, but to some other limit. The proof of such claims may involve complicated computations, which we choose to omit for sake of simplicity.
\end{remark}

\subsection{Finite rank perturbation for product of i.i.d. matrices}
To further illustrate the scope of Theorem \ref{spectralradiusproduct}, we show that the outliers of certain perturbed product matrices can also be determined without assuming finite fourth moments of $A_n$. Similar results were previously proven in \cite{MR4076784}, assuming a finite fourth moment.

\begin{theorem}\label{productoftheorem1} Assume that $C^1,\cdots,C^m$ are deterministic $n\times n$ square matrices that satisfy:
 each matrix $C^i$ has at most $M_1$ non-zero entries and all the nonzero entries of $C^i$ are bounded in norm by $M_2$, for $i=1,\cdots,m$, and constants $M_1,M_2$ are independent of $n$. 

We define a linearized version 
\begin{equation}
\mathcal{C}=\mathcal{C}_n=\begin{pmatrix}
&C^2&&&\\&&C^3&&\\&&&\ddots&\\&&&&C^m\\C^1&&&&
\end{pmatrix},\end{equation}
with the subscript $n$ indicating the dependence of $C^i$ on $n$, and we denote $c_n(z)=\det(1-z\mathcal{C}_n)$ the reverse characteristic function of $\mathcal{C}$. We assume that $c_n(z)$ converges to $c(z)$ as holomorphic functions on $\mathbb{D}$, where the limit is a fixed holomorphic function $c(z)$.

Let $A^1,\cdots,A^m$ be independent copies of i.i.d. random matrices as specified in Theorem \ref{spectralradiusproduct}. 
Then almost surely, as $n$ tends to infinity, the eigenvalues of the product matrix $$\prod_{i=1}^m(n^{-1/2}A^i+C^i),$$ with modulus greater than $1+o(1)$ converge to solutions $z\in\mathbb{C}\setminus\mathbb{D}$ to the equation $c(z^{-\frac{1}{m}})=0$.
\end{theorem}

\section{Spectral radius of product matrix with second moment}

This section is devoted to the proof of Theorem \ref{spectralradiusproduct}.

We will use a standard linearization procedure (used in a similar context in \cite{MR4076784}) to transfer the problem on the product matrix $A^1\cdots A^m$ to the problem on the linearized model $\mathcal{Y}$, where $\mathcal{Y}$ is defined as

\begin{equation}\label{welinearizeit}
\mathcal{Y}=\begin{pmatrix}
&A^2&&&\\&&A^3&&\\&&&\ddots&\\&&&&A^m\\A^1&&&&
\end{pmatrix}.\end{equation}

\begin{lemma}\label{theorem2.1345}
    Theorem \ref{spectralradiusproduct} is implied by the following statement: for any $\epsilon>0$, we have with probability $1-o(1)$, 
    $$\rho(\frac{1}{\sqrt{n}}\mathcal{Y})\leq 1+\epsilon.$$
\end{lemma}

\begin{proof}
    This readily follows from the fact that (see for example \cite{MR4076784}, Proposition 4.1) eigenvalues of $\frac{1}{\sqrt{n}}\mathcal{Y}$ raised to the $m$-th power coincide (with multiplicity $m$) to eigenvalues of the product $n^{-m/2}A^1\cdots A^m$.
\end{proof}

Therefore we are reduced to the matrix $\mathcal{Y}$ (we hide the subscript $n$ from the notation $\mathcal{Y}$) with independent entries. We will carry out the program in \cite{bordenave2021convergence} to the matrix $\mathcal{Y}$: first we define the characteristic functions 
$$
q_n(z)=\det(1-n^{-1/2}z\mathcal{Y}),\quad z\in\mathbb{C},
$$where $1$ is the unit square matrix of size $mn$. We now show $q_n(z)$ converges to a Gaussian analytic function on $(1-\epsilon)\mathbb{D}$.

The first step is to show tightness of $q_n(z)$. We shall rewrite $q_n(z)$ as

$$
q_n(z)=1+\sum_{k=1}^{mn}(-z)^k P_k^{nm},
$$
where 
$$
P_k^{nm}=\sum_{|I|=k,I\subset\{1,\cdots,mn\}}n^{-k/2}\det(\mathcal{Y}(I))
$$ where $\mathcal{Y}(I)$ is the principal submatrix of $\mathcal{Y}$ with only the rows and columns indexed by $I$.

We need a technical lemma on the convergence of random power series:
 
\begin{lemma}\label{lemma3.12345}(\cite{bordenave2021convergence}, Lemma 3.2)
Given $f_n(z)=\sum_{k=0}^\infty z^kf_k^n$ a sequence of power series, and assume that $\{f_n(z)\}$ are tight in $\operatorname{H}(\mathbb{D})$. Then suppose that for any $m\geq 0$,
\begin{equation}
    (f_0^n,\cdots,f_m^n)\overset{\text{law}}{\underset{n\to\infty}{\longrightarrow}}
 (f_0,\cdots,f_m),
\end{equation}
where $(f_0,\cdots,f_m,\cdots)$ is a sequence of common random variables, 
then $f=\sum_{k=0}^\infty z^k f_k$ is defined in $\operatorname{H}(\mathbb{D})$ and $f_n\overset{\text{law}}{\underset{n\to\infty}{\longrightarrow}} f$.
\end{lemma}

We first verify that our random power series is tight:

\begin{lemma}\label{tightnesslemma}
    The sequence $q_n(z)$ is tight in $H(\mathbb{D})$.
\end{lemma}

\begin{proof} By \cite{shirai2012limit}, Lemma 2.6
    It suffices to bound $\mathbb{E}[|q_n(z)|^2]$ via a deterministic function of $z$ which is continuous and $n$-independent. Since entries of $\mathcal{Y}$ are all independent, we have 
   $$\mathbb{E}\left[\det\mathcal{Y}(I)\overline{\det\mathcal{Y}(J)}\right]=0\quad\text{ if } I\neq J.
    $$ Estimating $\mathbb{E}[|\det\mathcal{Y}(I)|^2]$ needs more care: we can verify that $\mathbb{E}[|\det\mathcal{Y}(I)|^2]\neq 0$ only if $I$ has nonzero intersection with all the $m$ intervals $\{1,\cdots,n\},\{n+1,\cdots,2n\},\cdots,\{(m-1)n+1,\cdots,mn\}$. Indeed if $I$ has zero intersection with any such interval, then $\mathcal{Y}(I)$ has some columns that are completely zero. If we denote by $n_i=|I\cap \{(i-1)n+1,\cdots,in\}$ the cardinality of intersection of $I$ with any of these intervals and denote by $I_i$ the interval of intersection $I\cap \{(i-1)n+1,\cdots,in\}$, then 
    $$
\mathbb{E}[|\det\mathcal{Y}(I)|^2]=\prod_{k=1}^m (n_k)!,\quad\text{ and } \sum_{k=1}^m n_k=|I|,
    $$ and 
    $$
\mathbb{E}[|P_k^{mn}|^2]=n^{-k}\sum_{k=n_1+\cdots+n_m}\sum_{I_1,\cdots,I_m} \prod_{k=1}^m(n_k)!
    $$
where the first sum is over methods to write $k=n_1+\cdots+n_m$ as a sum over positive integers $n_1,\cdots,n_m$ and the second sum is over intervals $I_1,\cdots,I_m$ where each $I_i$ lies in $\{(i-1)n+1,\cdots,in\}$ and has cardinality $n_i$. 

Thus we can estimate via 
\begin{equation}\label{equation233332}
\mathbb{E}[|P_k^{mn}|^2]\leq n^{-k}C_{k-1}^{m-1}\sup_{(n_1,\cdots,n_m):\sum_{i=1}^m n_i=k}\prod_{i=1}^m\frac{n!}{(n_i)!(n-n_i)!}\prod_{k=1}^m(n_k)!
  \leq C_{k-1}^{m-1}.\end{equation}
Then we have, for any $|z|<1$,

$$
\mathbb{E}|q_n(z)|^2\leq1+\sum_{k\geq 0}|z|^kC_{k-1}^{m-1}<\infty,
$$and the bound can be written as a bounded continuous function in $|z|$. This completes the proof of tightness.
\end{proof}

Having verified tightness of $q_n$, we do two technical reductions to $q_n$. First, we outline that to prove the convergence of $q_n$ to some Gaussian analytic function it suffices to work in the case where entries of $\mathcal{Y}$ are bounded a.s. with a bound independent of $n$. This is the content of Lemma \ref{whatisnewlemma2.3}.

 Denote by $(y_{ij})_{1\leq i,j\leq mn}$ the matrix entries of $\mathcal{Y}$, so they are either zero or some $a_{ij}^k$.
 Given $M>0$ define $$\mathcal{Y}^M=\{y_{ij}^M\}_{1\leq i,j\leq n},\quad y_{ij}^M=y_{ij}1_{|y_{ij|\leq M}}-\mathbb{E}[y_{ij}1_{|y_{ij}|\leq M}].$$
Then defining $$P_k^{mn,M}=\sum_{I\subset\{1,2\cdots mn\}:|I|=k}n^{-k/2}\det\left((\mathcal{Y}^M(I)\right)),$$ we will need the following lemma:

\begin{lemma}\label{whatisnewlemma2.3}
Suppose that there exists a sequence of coefficients $(Y_1^M,\cdots,Y_k^M)$ such that 
$$ (P_1^{mn,M},\cdots,P_k^{mn,M})
\overset{\text{law}}{\underset{n\to\infty}{\longrightarrow}} (Y_1^M\cdots,Y_k^M),\quad \text{ and } (Y_1^M\cdots,Y_k^M)\overset{\text{law}}{\underset{M\to\infty}{\longrightarrow}} (Y_1,\cdots,Y_k), 
$$for all $M>0$, then we have 
$$
(P_1^{mn},\cdots,P_k^{mn})\overset{\text{law}}{\underset{n\to\infty}{\longrightarrow}} (Y_1,\cdots,Y_k).
$$
\end{lemma}    
This lemma can be proved similarly to \cite{bordenave2021convergence} , Lemma 3.3 but does not exactly follow from that proof as the structure of $\mathcal{Y}$ is different. We give a sketch:

\begin{proof}
For each fixed $k$ it suffices to check that we can find a sequence of positive numbers $C_M(k)\to 0$ as $M\to\infty$ such that 
    $$
\mathbb{E}[|P_j^{mn,M}-P_j^{mn}|^2]\leq C_M(k),\quad\text{ for all }1\leq j\leq k,
$$ uniformly in $n$.
    Now we bound
    $$\begin{aligned}
\mathbb{E}[|P_k^{mn,M}-P_k^{mn}|^2]&=n^{-k} \sum_{|I|=k,I\subset\{1,\cdots,mn\}} \mathbb{E}[|\det(\mathcal{Y}(I))-\det(\mathcal{Y}^M(I))|^2]\\&\leq C_{k-1}^{m-1}\mathbb{E}[|a_{11}^1[M]\cdots a_{1k}^1[M]-a_{11}^1\cdots a_{1k}^1|^2],
    \end{aligned}$$ where for $i,k\in\{1,2,\cdots,n\}$ and $j\in\{1,2,\cdots,m\}$ we denote by $$a_{ik}^j[M]:=a_{ik}^j1_{|a_{ik}^j|\leq M}-\mathbb{E}[a_{ik}^j1_{|a_{ik}^j|\leq M}],$$ and where we use the superscript $j$ in $a_{ik}^j$ to mean that $a_{ik}^j$ is an entry in the $j$-th matrix $A^j$. In the second line we use a similar combinatorial counting procedure as in \eqref{equation233332}, as only intervals $I$ having nonzero intersections with all blocks do contribute to the sum.

    Therefore we can take $C_M(k)$ to be $C_{k-1}^{m-1}\mathbb{E}[|a_{11}^1[M]\cdots a_{1k}^1[M]-a_{11}^1\cdots a_{1k}^1|^2]$ which converges to $0$ as $M\to\infty$.
\end{proof}

The second reduction, outlined after \cite{bordenave2021convergence}, Lemma 3.3, is that we can rewrite $q_n(z)$ as

\begin{equation}
q_n(z)=\exp\left(-\sum_{k=1}^\infty \frac{\operatorname{Tr}(\mathcal{Y}^k)}{n^{k/2}}\frac{z^k}{k}\right),
\end{equation} and to show the joint convergence of $(P_1^{mn},\cdots,P_k^{mn})$, it suffices to show the joint convergence of $\{\frac{\operatorname{Tr}(\mathcal{Y}^j)}{n^{j/2}}\frac{z^j}{j}\}_{j=1,\cdots,k}$, which can be more easily derived.

By the block structure of $\mathcal{Y}$ we have that $\operatorname{Tr}(\mathcal{Y}^k)=0$ if $k$ is not a multiple of $m$, so we consider only $k$ a multiple of $m$. We separately consider the terms in expansion $\operatorname{Tr}(\mathcal{Y}^k)$ that have $k$ distinct vertices in the expansion, and those that have fewer than $k$ indices.

As we expand the trace, a typical element in the trace $\operatorname{Tr}(\mathcal{Y}^k)$ expansion should be of the form $$a_{i_1,i_2}^{k_1}a_{i_2,i_3}^{k_1+1} a_{i_3,i_4}^{k_1+2}\cdots a_{i_ki_1}^{k_1+k-1}$$ 
for some $k_1\in\{1,2,\cdots,m\}$ and $i_1,i_2,\cdots,i_k\in\{1,2,\cdots,n\}$, and in the superscript we use $k_1+s$ to mean actually $k_1+s$ mod $m$, for any $s\in\mathbb{N}_+$.

This typical element in the expansion is uniquely indexed by $g=(k_1,i_1,i_2,\cdots,i_k)$ where we use the symbol $g$ as the collection of all these variables, and we use the notation
$$a_g:=a_{i_1,i_2}^{k_1}a_{i_2,i_3}^{k_1+1} a_{i_3,i_4}^{k_1+2}\cdots a_{i_ki_1}^{k_1+k-1},\quad g=(k_1,i_1,i_2,\cdots,i_k).
$$

Now we consider two subsets $\mathcal{G}_1(k)^{(n)}$ and $\mathcal{G}_2(k)^{(n)}$ that partition all possible choices of $g$, with the superscript $(n)$ means we take $i_1.\cdots,i_k\in\{1,2,\cdots,n\}$, and where for $g\in \mathcal{G}_1(k)^{(n)}$ we request that $\operatorname{Card}\{i_1,\cdots,i_k\}=k$, and for $g\in \mathcal{G}_2(k)^{(n)}$ we request that $\operatorname{Card}\{i_1,\cdots,i_k\}\leq k-1.$ When $k$ is not divisible by $m$ we set $\mathcal{G}_1(k)^{(n)}=\mathcal{G}_2(k)^{(n)}=\emptyset.$ The two subsets certainly also depend on $m$, but we omit them from the notation.

We further introduce the notation 
$$t_k:=\sum_{g\in \mathcal{G}_1(k)^{(n)}} a_g,\quad r_k:=\sum_{g\in \mathcal{G}_2(k)^{(n)}} a_g,$$ so that $$
\operatorname{Tr}(\mathcal{Y}^k)=t_k+r_k.
$$

We will verify in the following that $t_k$ contains the terms in $\operatorname{Tr}(\mathcal{Y}^k)$ that contribute to the randomness, and $r_k$ contains the terms contributing to a fixed expectation.

We first determine the distributional limit of $t_k$.

\begin{lemma} Fix any $k_1,\cdots,k_r\geq 1$ such that $k_1,\cdots,k_r$ are integer multiples of $m$. Fix any chain of symbols $s_1,\cdots,s_r\in \{\cdot,*\}$. Then we have
$$
\mathbb{E}\left[(\frac{t_{k_1}}{mk_1n^{k_1/2}})^{s_1}\cdots (\frac{t_{k_r}}{mk_rn^{k_r/2}})^{s_r}\right]\to_{n\to\infty}\mathbb{E}\left[(\frac{X_{k_1}}{\sqrt{mk_1}})^{s_1}\cdots (\frac{X_{k_r}}{\sqrt{mk_r}})^{s_r}\right]
$$with the convention that $x^\cdot=x$ and $x^*=\bar{x}$. The sequence $\{X_k\}_{k\geq 1}$ are independent complex Gaussian random variables with $\mathbb{E}[X_k]=0,$  $\mathbb{E}[|X_k|^2]=1$ and  $\mathbb{E}[X_k^2]=\mathbb{E}[(a_{11}^1)^2]^k.$
\end{lemma}

\begin{proof}
    The proof is similar to \cite{bordenave2021convergence}, Lemma 3.4. By Lemma \ref{whatisnewlemma2.3}, we only need to consider the case when all $a_{ij}^k$ are bounded a.s. by some fixed constant $M$, for $M$ sufficiently large.  As we expand the trace and take the union of the edges coming from each trace expansion, the restriction that each edge should be traveled at least twice (to guarantee a non-vanishing expectation) forces that there are at most $\frac{k_1+\cdots+k_r}{2}$ free vertices to choose. If some edges are traveled three times or more in the union, then the total choice of free vertices is no more than $\frac{k_1+\cdots+k_r}{2}-1$, and (since $a_{ij}$ are a.s. bounded) the contribution of such terms vanish in the limit. 

More precisely, we need to compute the moments of 
    $$
\frac{1}{n^{(k_1+\cdots+k_r)/2}} \sum_{g_1\in\mathcal{G}_1(k_1)^{(n)},\cdots,g_r\in\mathcal{G}_1(k_r)^{(n)}}\mathbb{E}[(a_{g_1})^{s_1}\cdots (a_{g_r})^{s_r}].
    $$
    We introduce an equivalence relation where $(g_1,\cdots,g_r)$ is equivalent to $(\tilde{g}_1,\cdots,\tilde{g}_r)$ if there exists a bijection $\theta:\{1,2,\cdots,m\}\times\{1,2,\cdots,n\}\to\{1,2,\cdots,m\}\times\{1,2,\cdots,n\}$ 
    such that $\tilde{g}_i=\theta_*(g_i)$ where $\theta_*$ is the mapping induced by $\theta$ on the labels. Since the value of the expectation of $\mathbb{E}[(a_{g_1})^{s_1}\cdots (a_{g_r})^{s_r}]$ is the same if we replace $(g_1,\cdots,g_r)$ by an equivalent element, if we denote by $\mathcal{G}^{(n)}_{(k_1,\cdots,k_r)}$ the set of equivalence classes, we can then define $W_{[\Gamma]}=\mathbb{E}[(a_{g_1})^{s_1}\cdots (a_{g_r})^{s_r}]$ for $[\Gamma]$ being the equivalence class of $\Gamma$. Then we can compute that 
    $$
\frac{1}{n^{(k_1+\cdots+k_r)/2}}\sum_{g_i\in\mathcal{G}_1(k_i)^{(n)},i\in[r]}\mathbb{E}[(a_{g_1})^{s_1}\cdots (a_{g_r})^{s_r}]=\frac{1}{n^{(k_1+\cdots+k_r)/2}}\sum_{\nu\in\mathcal{G}^{(n)}_{(k_1,\cdots,k_r)}}\operatorname{Card}(\nu)W_\nu,
    $$ and we use $\operatorname{card}(\nu)$ to denote the cardinality of $\nu$ in the product set $\mathcal{G}_1(k_1)^{(n)}\times\cdots\mathcal{G}_r(k_r)^{(n)}$. For any $\mu\in \mathcal{G}_{k_1,\cdots,k_r}^{(n)}$, we associate to it an oriented multigraph $G^\mu$ which is the union of the $g_l$'s with multiple edges. That is, let $V^{g_l}$ and $E^{g_l}$ denote the vertex set and edge set of $g_l$, then we define the vertex set $V^\mu$ and edge set $E^\mu$ of $G^\mu$ as 
    $$
V^\mu=\cup_{l=1}^m V^{g_l},\text{ and }E^\mu=\cup_{l=1}^m (\{l\}\times E^{g_l}). 
    $$Also set up a mapping $s:E^\mu\to V^\mu$, $t:E^\mu\to V^\mu$ via 
    $s(l,(i,j))=i,\quad t(l,(i,j))=j$. For each $v\in V^\mu$ define the outer degree $\deg(v)$ via $$\deg(v)=\operatorname{card}\{e\in E^\mu:s(e)=v\},$$ so that $\operatorname{deg}(v)\geq 2$ for each $v\in V^\mu$. One can also check that $\sum_{v\in V^\mu}\operatorname{deg}(v)=\operatorname{Card}(E^\mu)=k_1+\cdots+k_r$. If $\operatorname{deg}(v_*)\geq 3$ for one $v_*\in V^\mu$ then $\operatorname{card}(V^\mu)<\frac{k_1+\cdots+k_r}{2}.$
This implies that $\frac{\operatorname{card}_n(\mu)}{n^{(k_1+\cdots+k_r)/2}}\to_{n\to\infty}0$. 

Then assume that $\operatorname{deg}(v)=2$ for each $v\in V^\mu$. Then since each edge of $G^\mu$ is multiple and each vertex has degree exactly 2, there must be a partition into pairs of $\{1,2,\cdots,m\}$ into pairs and that 
$$
g_i\text{ is equal to }g_j\text{ up to a permutation of indices if }i\sim j,\quad V^{g_i}\cap V^{g_j}=\emptyset\text{ if } i\neq j,  
$$ where we say $g_i=(k_1,i_1,\cdots,i_k)$ is a permutation of $g_j=(k_1',i_1',\cdots,i_k')$ if there exists some $x\in\{1,2,\cdots,k\}$ such that $i_t'=i_{t+x}$ for each $t=1,\cdots,k$, and $k_1'=k_1+x$, with the summation $t+x$ is mod $k$ and the summation $k_1+x$ is mod $m$. 

Then necessarily $r$ is even and for this $\mu$ we have
$$
\frac{\operatorname{card}_n(\mu)}{n^{(k_1+\cdots+k_r)/2}}\to_{n\to\infty}m^{r/2}\sqrt{k_1\cdots k_r}
$$ where the first $m$ factor comes from choosing a block from the $m$ available blocks to begin with, and the remaining factors $k_1,\cdots,k_r$ come from first choosing the vertices of $g_i$, then choose an arbitrary ordering of $g_j$ relative to $g_i$ when $i\sim j$. There is a square root because each value among $k_1,\cdots,k_r$ will appear twice but only one of them will contribute to the overall multiplicity since for a pair $i\sim j$ with $k_i=k_j$, we choose $n^{k_i}(1+o(1))$ vertices for $g_i$ and then thee are $k_i$ choices for $g_j$ from a permutation of $g_i$.

    Then we can check via Wick's theorem the convergence stated in the lemma.
    This completes the proof. From the proof we see that the only terms contributing to the limit are those that form a disjoint union of cycles that contain only doubled edges, and this is where the i.i.d. assumption adds into more cancellations. 
\end{proof}

Then we determine the distributional limit of $r_k$.

\begin{lemma}
    For each $k\in\mathbb{N}_+$ we have the convergence in distribution 
    $$
\frac{r_k}{n^{k/2}}\to_{n\to\infty} \begin{cases}m\mathbb{E}[(a_{11}^1)^2]^{k/2},\quad \text{if $k$ is divisible by $2m$},\\ 0\quad\text{ else. }\end{cases}
    $$
\end{lemma}

\begin{proof} Again we may assume each entry $a_{ij}^k$ are bounded a.s. by some large constant $M$.
    First we prove that $\mathbb{E}[\frac{r_k}{n^{k/2}}]$ converges to the stated limit. For the expectation of an element $a_g$, $g\in\mathcal{G}_2(k)^{(n)}$ to be nonzero, the directed edges traveled by the elements of $g$ must each be traveled twice or more, so that there are at most $k/2$ choices of free vertices. If $a_g$ has some edges traveled in either direction at least three times, then we have at most $k/2-1$ free vertices, and the contribution of such terms $a_g$ to $\mathbb{E}[r_k]$ vanishes. Thus the only dominating contribution of $a_g$ to $\mathbb{E}[r_k]$ are double cycles, i.e. those paths $a_g$ which satisfy $i_t=i_q$ if and only if $q=t$ or $q=t+k/2$, and enumerating all double cycles gives a contribution in expectation $m\mathbb{E}[(a_{11}^1)^2]^{k/2}n^{k/2}(1+o(1))$ in the limit.

    Next we show $\operatorname{Var}(\frac{r_k}{n^{k/2}})\to 0$ as $n\to\infty.$ It suffices to compute $\sum_{g_i,g_j}\operatorname{Cov}(a_{g_i},a_{g_j})$ over all pairs of $(g_i,g_j)\in\mathcal{G}_2(k)^{(n)}$ such that in the expression $a_{g_i}$ and $a_{g_j}$ there is at least one common element, otherwise the covariance is 0 because of independence. Since all entries $a_{ij}^k$ are bounded by $M$, say, we have
    $$
|\operatorname{Cov}(a_{g_i},a_{g_j})|\leq 4M^{2k}
    $$  
    is bounded independent of $n$. For an element $a_{g_i}$ where $g_i=(k_1,i_1,\cdots,i_k)$ we define $V_i=\{(k_1-1)m+i_1,k_1m+i_2,\cdots,(k_1+k-1)m+i_k\}$ as the corresponding labels on $[mn]\times[mn]$, where the summation is mod $mn$, and  $E_i:=\{((k_1+s-1)m+i_s,(k_{1}+s)m+i_{s+1}),s=1,2,\cdots,k-1\}$. Let $G_i=(V_i,E_i)$ denote the directed graph associated to $g_i$. Define the excess of $G_i$ as $\chi_i=\operatorname{card}(E_i)-\operatorname{card}(V_i)+1\geq 0$, the smallest number of edges to be removed guaranteeing that the remaining subgraph does not have an undirected cycle. The assumption $\operatorname{card}(V_i)\leq k$ implies $\chi_i\geq 2$. 
    
Using this construction, for two tuples $i$ and $j$, define the associated graph with vertices and directed edges as follows:
$$
V_{ij}=V_i\cup V_j,\quad E_{ij}=E_i\cup E_j.$$
The excess for the graph $G_{ij}=(V_{ij},E_{ij})$ is defined as $\chi_{ij}=\operatorname{card}(E_{ij})-\operatorname{card}(V_{ij})+c$ with $c=1$ if $V_i\cap V_j\neq\emptyset$ and $c=2$ otherwise. We have $\chi_{ij}\geq\max(\chi_i,\chi_j)\geq 2$. To guarantee a nonzero covariance, we need that each edge of $E_{ij}$ are visited at least twice, so $\operatorname{card}(E_{ij})\leq k$ and thus $\operatorname{card}(V_{ij})=1-\chi_{ij}+\operatorname{Card}(E_{ij})\leq k-1$. This implies that 
$$
\operatorname{var}(r_k^{(n)})\leq 4C_kM^{2k}n^{k-1}
$$ where $C_k$ is the number of possibilities for choosing the $k$-tuples $i,j$ once $V_{ij}$ is fixed. This verifies that $\operatorname{Var}(\frac{r_k}{n^{k/2}})\to 0$ and completes the proof.
\end{proof}

Now we can complete the proof of Theorem \ref{spectralradiusproduct}.

\begin{proof}[\proofname\ of Theorem \ref{spectralradiusproduct}]
From the previous three lemmas we conclude that 
$$\begin{aligned}&
\left(\frac{\operatorname{Tr}(\mathcal{Y}^m)}{n^{m/2}},\frac{\operatorname{Tr}(\mathcal{Y}^{2m})}{n^{2m/2}},\cdots \frac{\operatorname{Tr}(\mathcal{Y}^{km})}{n^{km/2}}\right)\\&\to_{law} (\sqrt{m^2}X_m,\sqrt{2m^2}X_{2m},\cdots \sqrt{km^2}X_{km})+(\text{mean}_m,\text{mean}_{2m},\cdots, \text{mean}_{km}).
\end{aligned}$$
    where for each $k\in\mathbb{N}_+$ ($X_{km}$) are independent Gaussian variables with mean zero variance one and $\mathbb{E}[X_{km}^2]=\mathbb{E}[a_{11}^2]^{km}$, and for each $k\in\mathbb{N}_+$ $\text{mean}_{km}=\frac{1}{2}(1+(-1)^k)m\mathbb{E}[(a_{11}^1)^2]^{km/2}$. 
Since $q_n$ is tight by Lemma \ref{theorem2.1345}, and the coefficients of $q_n$ converge to specified limits, by Lemma \ref{lemma3.12345}, the sequence $q_n$ also converges in distribution to the specified limiting object.

Then the above shows that 
\begin{equation}\label{expressionqns}q_n
\overset{\text{law}}{\underset{n\to\infty}{\longrightarrow}} \kappa e^{-F},\end{equation}
where 
\begin{equation}\label{productm}F(z)=F_{\eqref{productm}}(z)=\sum_{k\in\mathbb{N}_+,k\text{ divisible by }m} X_k z^k\sqrt{\frac{m}{k}}\end{equation}
and 
\begin{equation}\label{productm2s}\kappa(z)=
\kappa_{\eqref{productm2s}}(z)=\exp\left(-\sum_{k\in\mathbb{N}_+,k\text{ divisible by }2m} z^k\frac{m\mathbb{E}[(a_{11}^1)^2]^{km/2}}{k}\right).
\end{equation}
For any $\epsilon>0$, it is elementary to check that $\kappa(z)$ has no zeros for $z\in(1-\epsilon)\mathbb{D}$ and $F(z)$ has no zeros almost surely for $z\in(1-\epsilon)\mathbb{D}$.
Since $q_n$ converges to $q$ in distribution, by Rouche's theorem with probability $1-o(1)$ $q_n(z)$ has no zeros in $(1-\epsilon)\mathbb{D}$. This implies $\rho(\mathcal{Y})\leq \frac{1}{1-\epsilon}$ by definition of $q_n$, and concludes the proof of Theorem \ref{spectralradiusproduct} via Lemma \ref{theorem2.1345}.
\end{proof}

\section{Perturbation for an i.i.d. matrix with only finite second moment}

We begin with some very general reduction techniques that apply to any deterministic $C_n$ satisfying the assumption of Theorem \ref{theorem1}. In this case the characteristic function has the following expansion:
$$
q_n(z)=\det(1-z(\frac{A_n}{\sqrt{n}}+C_n))=1+\sum_{k=1}^n (-z)^kP_k^{n}
$$
where 

$$
P_k^{n}=\sum_{I\subset \{1,2,\cdots,n\}:|I|=k}n^{-k/2}\det\left((A_n+\sqrt{n}C_n)(I)\right)
$$
with $(A_n+\sqrt{n}C_n)(I)=\{a_{jk}+\sqrt{n}c_{jk}\}_{j,k\in I}$.

The key point here (adopted from \cite{bordenave2021convergence}) is that, once we can show $q_n(z)$ is tight as an element in $\operatorname{H}(\mathbb{D})$, then the deduction of convergence of $q_n(z)$ can be obtained from the convergence of individual coefficients $P_k^{n}$. We note that it seems easier to check the convergence of individual coefficients $P_k^n$ and the tightness of $q_n$, instead of proving directly that $q_n$ converges in $H(\mathbb{D}).$ 

A further technical reduction is truncation to the case where all the $a_{ij}$ are bounded. Given $M>0$ define $$A_n^M=\{a_{ij}^M\}_{1\leq i,j\leq n},\quad a_{ij}^M=a_{ij}1_{|a_{ij|\leq M}}-\mathbb{E}[a_{ij}1_{|a_{ij}|\leq M}].$$
Then defining $$p_k^{n,M}=\sum_{I\subset\{1,2\cdots n\}:|I|=k}n^{-k/2}\det\left((A_n^M+n^{1/2}C_n)(I)\right),$$ we will need the following lemma:

\begin{lemma}\label{greatlemma2.34}
Suppose that there exists a sequence of coefficients $(Y_1^M,\cdots,Y_k^M)$ such that 
$$ (P_1^{n,M},\cdots,P_k^{n,M})
\overset{\text{law}}{\underset{n\to\infty}{\longrightarrow}} (Y_1^M\cdots,Y_k^M),\quad \text{ and } (Y_1^M\cdots,Y_k^M)\overset{\text{law}}{\underset{M\to\infty}{\longrightarrow}} (Y_1,\cdots,Y_k), 
$$for all $M>0$, then we have 
$$
(P_1^n,\cdots,P_k^n)\overset{\text{law}}{\underset{n\to\infty}{\longrightarrow}} (Y_1,\cdots,Y_k).
$$
\end{lemma}
This lemma in the $C_n=0$ case is already proved in \cite{bordenave2021convergence}, Lemma 3.3. The proof in the general case is exactly the same as $C_n$ is just a deterministic matrix with bounded entries. The proof is left to the reader.

Now we use the reduction from \cite{bordenave2021convergence}which was already highlighted in the introduction: at least when $z$ is sufficiently small, we have that 
$$B_n:=-\sum_{k=1}^\infty (\frac{A_n+\sqrt{n}C_n}{\sqrt{n}})^k\frac{z^k}{k}$$
is well-defined and converges to the logarithm of $1-z\frac{A_n+\sqrt{n}C_n}{\sqrt{n}}.$ From the matrix relation $\det(e^{B_n})=e^{\operatorname{Tr}B_n}$, this leads to, for $z$ sufficiently small, 
$$
q_n(z)=\exp\left(-\sum_{k=1}^\infty \frac{\operatorname{Tr}((A_n+\sqrt{n}C_n)^k)}{n^{k/2}}\frac{z^k}{k}
\right).
$$ Since two analytic functions that are equal in a small neighborhood of complex plane must be equal identically, this leads to the following consequence:

For each $k\in\mathbb{N}_+$, the coefficients $(P_1^n,\cdots,P_k^n)$ can be represented as ($n$-independent) polynomial functions of $$\left(\frac{\operatorname{Tr}((A_n+\sqrt{n}C_n))}{\sqrt{n}},\cdots,\frac{\operatorname{Tr}((A_n+\sqrt{n}C_n)^k)}{n^{k/2}}\right)$$ (and in reverse). Thanks to this expression, all the further effort of proving convergence of $q_n(z)$ boils down to proving the convergence of $\operatorname{Tr}(A_n+\sqrt{n}C_n)^k/n^{k/2}$ as $n$ tends to infinity.

\subsection{Illustrative example: a diagonal matrix of rank one}

For illustrative purpose, we begin with the proof of Theorem \ref{theorem1} in the simplest case where $C_n$ is a diagonal matrix of rank one. Then without loss of generality we may assume that $C_n=\theta E_{1,1}$ for some $\theta\neq 0$, where $E_{11}$ is the $n$ by $n$ matrix with 1 in its entry $(1,1)$ and $0$ everywhere else. We expand the trace as follows:
\begin{equation}\label{traceexpansion}\begin{aligned}
\operatorname{Tr}((A_n+\sqrt{n}C_n)^k)&=\sum_{\substack{1\leq i_1,i_2,\cdots, i_k\leq n,\\|\{i_1,\cdots,i_k\}|=k}}a_{i_1i_2}a_{i_2i_3}\cdots a_{i_ki_1}+\sum_{\substack{1\leq i_1,i_2,\cdots i_k\leq n,\\ |\{i_1,\cdots,i_k\}|\leq k-1}} a_{i_1i_2}a_{i_2i_3}\cdots a_{i_ki_1}\\
&+\sum_{\ell=1}^{k-1}(\sqrt{n}\theta)^\ell 
\sum_{\substack{1\leq i_1,i_2,\cdots, i_{k-\ell}\leq n,\\|\{i_1,\cdots,i_{k-\ell}\}|=k-\ell,\\ 1\in \{i_1,\cdots,i_{k-\ell}\}}}a_{i_1i_2}a_{i_2i_3}\cdots a_{i_{k-\ell} i_1}\\&+\sum_{\ell=1}^{k-1}(\sqrt{n}\theta)^\ell 
\sum_{\substack{1\leq i_1,i_2,\cdots, i_{k-\ell}\leq n,\\|\{i_1,\cdots,i_{k-\ell}\}|<k-\ell,\\ 1\in \{i_1,\cdots,i_{k-\ell}\}}}a_{i_1i_2}a_{i_2i_3}\cdots a_{i_{k-\ell} i_1}\\&+(\sqrt{n}\theta)^k.
\end{aligned}
\end{equation}

Note that thanks to Lemma \ref{greatlemma2.34}, we may assume that $a_{ij}$ are bounded and $|a_{ij}|\leq M$ almost surely.

Let $\mathfrak{C}_k^n$ be the set of $k$-directed cycles in $\{1,2,\cdots,n\}$, where a $k$-directed cycle is defined as a directed cycle of length $k$ on indices $1,2,\cdots,n$ having $k$ different vertices. Given $g\in \mathfrak{C}_k^n$ denote by $a_g=\prod_{e \text{ edge of }g} a_e$. Then the summation
$$
\sum_{\substack{1\leq i_1,i_2,\cdots, i_k\leq n,\\|\{i_1,\cdots,i_k\}|=k}}a_{i_1i_2}a_{i_2i_3}\cdots a_{i_ki_1}
$$
can be restated as
$$
t_k^n=\sum_{g\in\mathfrak{C}_k^n} a_g.
$$

The joint distribution of $t_k^n$ can be computed as in \cite{bordenave2021convergence}, Lemma 3.4:

\begin{lemma}(\cite{bordenave2021convergence}, Lemma 3.4)\label{fugagqgq}
    Given any $k_1,\cdots,k_m\geq 1$ and any sequence $s_1,\cdots,s_m\in\{\cdot,*\}$:
    \begin{equation}
        \mathbb{E}\left[\left(\frac{t_{k_1}^n}{n^{k_1/2}}\right)^{s_1}\cdots \left(\frac{t_{k_m}^n}{n^{k_m/2}}\right)^{s_m}
        \right]\to \mathbb{E}\left[\left(\frac{X_{k_1}}{\sqrt{k_1}}\right)^{s_1}\cdots\left(\frac{X_{k_m}}{\sqrt{k_m}}\right)^{s_m}
        \right],
    \end{equation} with the convention that $x^\cdot=x$ and $x^*=\bar{x}$,
    and that $\{X_k\}_{k\geq 1}$ are a sequence of independent complex Gaussian random variables with $\mathbb{E}[X_k]=0$, $
    \mathbb{E}[|X_k|^2]=1$ and $\mathbb{E}[X_k^2]=\mathbb{E}[a_{11}^2]^k$ for any $k\geq 1$.
\end{lemma}

Now we analyze the term on the second line of \eqref{traceexpansion}. Let $\mathfrak{C}_{k-\ell,1}^n$ be the set of $k-\ell$-directed cycles in $\{1,2,\cdots,n\}$ that contain 1. Given $g\in \mathfrak{C}_{k-\ell}^n$ denote by $a_g=\prod_{e \text{ edge of }g} a_e$. Then for the second line of \eqref{traceexpansion}, it suffices to study, for each $\ell=1,\cdots,k-1$: 
$$
t_{k,1}^{\ell,n}=(\sqrt{n}\theta)^\ell \sum_{g\in\mathfrak{C}_{k-\ell,1}^n} a_g.
$$

We will now state a series of lemmas and prove them in the sequel. We begin with

\begin{lemma}\label{lemmas2.4!}
    For each $k$ and each $\ell=1,2,\cdots,k-1$, we have the convergence in probability
    \begin{equation}
        \frac{t_{k,1}^{\ell,n}}{n^{k/2}}\overset{\text{law}}{\underset{n\to\infty}{\longrightarrow}} 0.
    \end{equation}
\end{lemma}

Having determined the limiting random term, at this time it suffices to consider the limiting constant term

$$
r_k^n=\sum_{\substack{1\leq i_1,\cdots,i_k\leq n\\ |\{i_1,\cdots,i_k\}|<k}}a_{i_1i_2}a_{i_2i_3}\cdots a_{i_{k-1}i_k}a_{i_ki_1}.
$$
and for each $\ell=1,\cdots,k-1$, define
$$
r_{k,1}^{\ell,n}=(\sqrt{n}\theta)^\ell  \sum_{\substack{1\leq i_1,i_2,\cdots, i_{k-\ell}\leq n,\\|\{i_1,\cdots,i_{k-\ell}\}|\leq k-\ell-1,\\1\in \{i_1,\cdots,i_{k-\ell}\}}}a_{i_1i_2}a_{i_2i_3}\cdots a_{i_{k-\ell}i_1}.
$$
For $r_k^n$, we have the following deterministic limit
from \cite{bordenave2021convergence}: 
\begin{lemma}\label{curious}(\cite{bordenave2021convergence}, Lemma 3.5) We have the following deterministic limit
$$ \frac{r_k^n}{n^{k/2}}\overset{\text{law}}{\underset{n\to\infty}\longrightarrow}     \begin{cases}
    \mathbb{E}[a_{11}^2]^{k/2},\quad  k \text{ even },
    \\ 0, \quad k \text{ odd }.
    \end{cases}$$
\end{lemma}

For $r_{k,1}^{\ell,n}$ we will also prove a similar result:
\begin{lemma}\label{lemma2.6!} We have the following convergence for each $\ell=1,\cdots,k-1$:
    $$ \frac{r_{k,1}^{\ell,n}}{n^{k/2}}\overset{\text{law}}{\underset{n\to\infty}\longrightarrow}    0.$$
\end{lemma}

\begin{proof}[\proofname\ of Lemma \ref{lemmas2.4!}] It suffices to expand
\begin{equation}\label{fucsg}n^{-k}\mathbb{E}[|t_{k,1}^{\ell,n}|^2]
\end{equation} and compute the expectation. For multi-indices $\mathbf{i}=(i_1,\cdots,i_{k-\ell})$, $\mathbf{j}=(j_1,\cdots,j_{k-\ell})$ in $\mathfrak{C}_{k-\ell,1}^n$, define $a_\mathbf{i}=a_{i_1i_2}a_{i_2i_3}\cdots a_{i_{k-\ell}i_i}$ . Then in order for $\mathbb{E}[a_\mathbf{i}a_\mathbf{j}]\neq 0$, one must have that the union of the edges traversed by $\mathbf{i}$ and $\mathbf{j}$ must each be traversed at least twice, because each $a_{ij}$ has mean zero. 

If there is an edge traversed more than twice by the union $\mathbf{i}\cup\mathbf{j}$, then by the handshaking lemma, the union of indices contained in $\mathbf{i}\cup \mathbf{j}$ is less than $k-\ell$, and except the index 1 already fixed, there are strictly less than $k-\ell-1$ number of distinct indices, and the contribution to the overall expectation \eqref{fucsg} of these terms is less than $O(n^{-k}n^{k-\ell-2}n^\ell \theta^{2\ell}(4M)^{2k})=O(n^{-2})$ (where $M$ is the almost sure bound on $a_{ij}$ thanks to Lemma \ref{greatlemma2.34}) and hence this term is negligible.

In the remaining case, each edge in the union of $\mathbf{i}\cup\mathbf{j}$ (multi-edges counted multiple times) are traversed exactly twice, and hence via a graphical reasoning, $\mathbf{i}\cup \mathbf{j}$ must be a doubled circle with $k-\ell$ edges and at most $k-\ell-1$ indices different from 1. The contribution of these terms to the expression \eqref{fucsg} is at most 
$$
n^{-k} n^{k-\ell-1}n^{\ell}\theta^{2\ell}(4M)^{2k}=O(n^{-1}),
$$ and hence the lemma is proved.
\end{proof}

\begin{proof}[\proofname\ of Lemma \ref{lemma2.6!}]
    We first prove that 
     $$\mathbb{E}\left[ \frac{r_{k,1}^{\ell,n}}{n^{k/2}}\right]{\underset{n\to\infty}\longrightarrow}     0.$$
     To see this, observe that for a cycle to contribute to the sum in expectation, from the cycle $(i_1,i_2,\cdots,i_{k-\ell})$ we construct a multigraph on $i_1,\cdots,i_{k-\ell}$ and every edge has to be repeated twice for it to contribute to the expectation. Next, if some edge has outer degree three, the contribution to the limit of such graphs is negligible by a similar reasoning as in previous lemma. Thus for the sum to contribute, the multigraph must be a multicircle and $k-\ell$ has to be even. Regardless of the parity of $k-\ell$, the number of circles of length $\frac{k-\ell}{2}$ that has 1 in its indices is at most $n^{\frac{k-\ell}{2}-1}$. Thus $\left|\mathbb{E}[r_{k,1}^{\ell,n}]\right|\leq n^{\frac{k}{2}-1}$ and hence the claim is proved. 

We next prove that $\operatorname{Var}\left[\frac{r_{k,1}^{\ell,n}}{n^{k/2}}\right]{\underset{n\to\infty}\longrightarrow}    0.$ We follow similar steps as in the proof of \cite{bordenave2021convergence},Lemma 3.5. For any $\mathbf{i}=(i_1,\cdots,i_{k-\ell})$, denote by $a_\mathbf{i}=a_{i_1i_2}a_{i_2i_3}\cdots a_{i_{k-\ell}a_{i_i}}$, we have trivially the bound 
$$
\mathbb{E}[(a_\mathbf{i}-\mathbb{E}(a_\mathbf{i}))(a_\mathbf{j}-\mathbb{E}(a_\mathbf{j}))]\leq 4M^{2k} 
$$ for any constant $M$ larger than the modulus of the support of $a_{11}$. 

Meanwhile, for any index $\mathbf{i}$, we regard $\mathbf{i}$ as a path of length $k$ and define a graph $(V_\mathbf{i},E_\mathbf{i})$ where $V_\mathbf{i}=\{i_1,\cdots,i_{k-\ell}\}$ and $E_\mathbf{i}=\{(i_r,i_{r+1}), r=1,\cdots,k-\ell\}$. 

Define the excess of graph $(V_\mathbf{i},E_\mathbf{i})$ as:
$$
\chi_\mathbf{i}=\operatorname{Card}(E_\mathbf{i})-\operatorname{Card}(V_\mathbf{i})+1\geq 0,
$$ which is the maximal number of edges to be removed from $(V_\mathbf{i},E_\mathbf{i})$ so that the resulting graph has no undirected cycles.
Then by definition $\chi_\mathbf{i}\geq 2$. Consider another multi-index $\mathbf{j}$ satisfying the same assumption, then $\chi_\mathbf{j}\geq 2$. We join the graphs $(V_\mathbf{i},E_\mathbf{i})$ and $(V_\mathbf{j},E_\mathbf{j})$ as follows: 
$$
V_{\mathbf{i}\mathbf{j}}=V_\mathbf{i}\cup V_\mathbf{j},\quad E_{\mathbf{i}\mathbf{j}}=E_\mathbf{i}\cup E_\mathbf{j}.
$$
Then defining $\chi_{\mathbf{i}\mathbf{j}}=\operatorname{Card}(E_{\mathbf{i}\mathbf{j}})-\operatorname{Card}(V_{\mathbf{i}\mathbf{j}})+c $ with $c=1$ when $V_\mathbf{i}\cap V_\mathbf{j}\neq \emptyset$ and $2$ otherwise, we deduce that $\chi_{\mathbf{i}\mathbf{j}}\geq \max(\chi_\mathbf{i},\chi_\mathbf{j})\geq 2$.

In order for $\mathbb{E}\left[(a_\mathbf{i}-\mathbb{E}(a_\mathbf{i}))(a_\mathbf{j}-\mathbb{E}(a_\mathbf{j}))\right]$ to be non-zero, we must require $V_\mathbf{i}\cap V_\mathbf{j}\neq\emptyset$ and $E_\mathbf{i}\cap E_\mathbf{j}\neq\emptyset$. Moreover, each edge in $E_{\mathbf{i}\mathbf{j}}$ must be traversed at least twice to guarantee a nonzero contribution in expectation. This leads to $\operatorname{Card}(E_{\mathbf{i}\mathbf{j}})\leq k-\ell$ and hence 
$$
\operatorname{Card}(V_{\mathbf{i}\mathbf{j}})=1-\chi_{\mathbf{i}\mathbf{j}}+\operatorname{Card}(E_{\mathbf{i}\mathbf{j}})\leq k-\ell-1.
$$
To sum up, we obtain 
$$
\operatorname{Var}(r_{k,1}^{\ell,n})\leq 4\theta^2 C_k M^{2k}n^{k-\ell-1}n^{\ell}=O(4\theta^2  M^{2k}n^{k-1})
$$
modulo universal constants depending only on $k$. This completes the proof.
\end{proof}

Now we are in the position to complete the proof of Theorem \ref{theorem1} in the special case $C_n=\theta E_{11}$.

\begin{proof}[\proofname\ of Theorem \ref{theorem1} for rank 1 diagonal noise] We should first check that the sequence of power series $q_n(z)$ are tight in $\operatorname{H}(\mathbb{D})$: this is already verified in \cite{bordenave2021convergence}, Lemma 3.1 when $C_n=0$, and it is easy to verify that this tightness persists when we add the noise matrix $C_n$.

Combining all the previous technical lemmas, in particular Lemma \ref{fugagqgq}, \ref{lemmas2.4!}, \ref{curious}, \ref{lemma2.6!}, we conclude that we have the following convergence as holomorphic functions in $\operatorname{H}(\mathbb{D})$:
\begin{equation}
    q_n(z)\overset{\text{law}}{\underset{n\to\infty}\longrightarrow}   (1-z\theta)\kappa(z)e^{-F(z)}
\end{equation}
where $\kappa(z)=\sqrt{1-z^2\mathbb{E}[a_{11}^2]}$ for $z\in\mathbb{D}$ and 
$$F(z)=\sum_{k=1}^\infty X_k\frac{z^k}{k} 
$$ where $\{X_k\}_{k\geq 1}$ are independent complex Gaussian variables with 
$$
\mathbb{E}[X_k]=0,\quad \mathbb{E}[|X_k|^2]=1,\quad \mathbb{E}[X_k^2]=\mathbb{E}[a_{11}^2]^k.
$$
We have also used 
$$
\exp\left(-\sum_{k=1}^\infty\frac{\theta^kz^k}{k}\right)=\exp\left(\log(1-z\theta)\right)=1-z\theta.
$$
By definition, $\kappa$ has no zeros in $\mathbb{D}$ since $|\mathbb{E}[a_{11}^2]|\leq 1.$
Since $F(z)$ converges almost surely for any $z\in\mathbb{D}$, we deduce that $e^{-F(z)}$ has no zeros in $\mathbb{D}$. That is, provided $|\theta|>1,$ the only zero of $(1-z\theta)\kappa(z)e^{-F(z)}$ in $\mathbb{D}$ is at $z=\frac{1}{\theta}$.

Since $q_n^\theta(z)$ converges almost surely to $q_n(z)$ in $\mathbb{H}(\mathbb{D})$, by Rouché's theorem, for $n$ large, almost surely there is only one solution $z_n\in\mathbb{D}$ such that $q_n(z_n)=0$ and that $\lim_{n\to\infty}z_n=\frac{1}{\theta}$. In other words, almost surely for $n$ sufficiently large, there is only one eigenvalue of $n^{-1/2}A_n+C_n$ with modulus larger than $1+o(1),$ and this eigenvalue converges to $\theta$ as $n$ tends to infinity. This proves Theorem \ref{theorem1} in this special case $C_n=\theta E_{11}$.
\end{proof}

\subsection{The general case of finitely many nonzero elements} Now we prove the general case of Theorem \ref{theorem1}, case (1) that covers all $C_n$ which has finite rank and is in a Jordan block form.

\begin{proof}[\proofname\ of Theorem \ref{theorem1},case 1] Consider again the trace expansion
\begin{equation}\begin{aligned}
\operatorname{Tr}((A_n+\sqrt{n}C_n)^k)&=\sum_{\substack{1\leq i_1,i_2,\cdots,i_k\leq n,\\|\{i_1,\cdots,i_k\}|=k}}(a_{i_1i_2}+\sqrt{n}C_{i_1i_2})(a_{i_2i_3}+\sqrt{n}C_{i_2i_3})\cdots (a_{i_ki_1}+\sqrt{n}C_{i_ki_1})
\\&+\sum_{\substack{1\leq i_1,i_2,\cdots,i_k\leq n,\\ |\{i_1,\cdots,i_k\}|\leq k-1}} (a_{i_1i_2}+\sqrt{n}C_{i_1i_2})(a_{i_2i_3}+\sqrt{n}C_{i_2i_3})\cdots (a_{i_ki_1}+\sqrt{n}C_{i_ki_1})\\
\end{aligned}    \end{equation}
We expand all the brackets on the right hand side, and denote by $T_1$ the terms consisting of monomials including both terms from $a_{ij}$ and $C_{ij}$ and that $|\{i_1,\cdots,i_k\}|=k$, and let $T_2$ consist of the monomials including both terms from $a_{ij}$ and $C_{ij}$ and that $|\{i_1,\cdots,i_k\}|\leq k-1$. We wish to prove that $\frac{T_1}{n^{k/2}}$ converges to $0$ in probability, and that $\frac{T_2}{n^{k/2}}$ converges to $0$ in probability.

To check the claim concerning $T_1$, we compute $\mathbb{E}[|T_1|^2]$. For two typical terms $w_{\mathbf{i}}=w_{i_1i_2}\cdots w_{i_ki_1}$ and $r_{\mathbf{j}}=r_{j_1j_2}\cdots r_{j_kj_1}$ (where each $w_{ij}$ and $r_{ij}$ are either some $a_{ij}$ or some $\sqrt{n}C_{ij}$), in order for $\mathbb{E}[w_\mathbf{i}r_\mathbf{j}]\neq 0$, we require that each edge traversed by some $a_{ij}$ must be traversed twice (those traversed more than twice has negligible contribution), and this forces there are exactly the same number of pairs in $w_\mathbf{i}$ and $r_\mathbf{j}$ that are some $a_{ij}$, and that they must correspond to the same $a_{ij}$ so as to appear at least twice. In this vein, assume that there are $Q$ edges in $w_\mathbf{i}$ and $r_\mathbf{j}$ traversed by some $A_n$. We see that to reconstruct this graph we actually need no more than $Q-1$ free vertices chosen from any of $1,2,\cdots,n$, as all other vertices are forced to be chosen from the subscript $i,j$ such that $C_{ij}\neq 0$, which are finite by definition. Altogether, there are $2(k-Q)$ other edges traversed by some $\sqrt{n}C_{ij}.$ Thus the overall contribution is $O(n^{Q-1}(\sqrt{n})^{2(k-Q)}),$ summing over these $\mathbf{i}\cup\mathbf{j}$. This readily implies $\frac{T_1}{n^{\frac{k}{2}}}\overset{\text{law}}{\underset{n\to\infty}\longrightarrow} 0.$

To check the claim concerning $T_2$, in order for a term $w_{i_1i_2}w_{i_2i_3}\cdots w_{i_ki_1}$ to have non-zero mean, each edge traversed by some $a_{ij}$ must be traversed twice, so the graph formed by $(i_1,\cdots,i_k)$ can be decomposed as (a) a double cycle containing all the terms from $A_{ij}$ and some other terms from $\sqrt{n}C_n$, and (b) some other cycles around it containing only elements from $\sqrt{n}C_n$. Let $R$ be the number of edges from $\sqrt{n}C_n$, then there are at most $\frac{k-R}{2}$ edges for the $A_{ij}$ terms (as each such edge should be traversed at least twice). This corresponds to $O(n^{\frac{k-R}{2}-1})$ choices of free vertices from $1,2,\cdots,n$ and a bounded number of choices from vertices $i,j$ such that $C_{i,j}\neq 0$. Here we have $O(n^{\frac{k-R}{2}-1})$ instead of $O(n^{\frac{k-R}{2}})$ choices because at least one vertices from the $a_{ij}$'s should come from the finitely many subscripts where $C_{ij}\neq 0$. The overall contribution to $\mathbb{E}[T_2]$ is thus $O(n^{\frac{k-R}{2}-1}n^{\frac{R}{2}})$ and we have that $\mathbb{E}[T_2]/n^\frac{k}{2}$ vanishes in the limit. 

Finally we check $\operatorname{Var}(T_2)=o(n^{k})$ as follows: the computation is exactly the same as before. For two typical terms $w_{\mathbf{i}}=w_{i_1i_2}\cdots w_{i_ki_1}$ and $r_{\mathbf{j}}=r_{j_1j_2}\cdots r_{j_kj_1}$ to have non-zero covariance, each edge traversed by some $a_{ij}$ should be traversed twice, and assuming there are $Q$ such edges, this leads to $Q-1$ free vertices and $2(k-Q)$ other edges traversed by $\sqrt{n}C_n$. This implies $\operatorname{Var}(T_2)=O(n^{k-1})$. Combining both steps we conclude that $\frac{T_2}{n^{k/2}}\overset{\text{law}}{\underset{n\to\infty}\longrightarrow}0$.

To sum up, using also the fact that $c_n(z)$ converges to $c(z)$, and checking the tightness of $q_n(z)$ just as in \cite{bordenave2021convergence}, Lemma 3.1, we have proved that 
\begin{equation}
    \lim_{n\to\infty }q_n(z)=\kappa(z) c(z)e^{-F(z)}\quad z\in\mathbb{D},
\end{equation} in the sense of convergence as functions in $\operatorname{H}(\mathbb{D})$, and the proof of Theorem \ref{theorem1}, case (1) is complete thanks to Rouche's theorem.
\end{proof}

\subsection{Infinitesimally shifting the mean} In this case we consider the matrix $C_n$ such that 
$$
C_{ij}=\frac{1}{n}\mu
$$ for each $i,j=1,\cdots,n$.
In this case we have
\begin{equation}\label{traceexpanecondjor}\begin{aligned}
\operatorname{Tr}((A_n+\sqrt{n}C_n)^k)&=\sum_{\substack{1\leq i_1,i_2,\cdots, i_k\leq n,\\|\{i_1,\cdots,i_k\}|=k}}(a_{i_1i_2}+\frac{\mu}{\sqrt{n}})(a_{i_2i_3}+\frac{\mu}{\sqrt{n}})\cdots (a_{i_ki_1}+\frac{\mu}{\sqrt{n}})\\&+\sum_{\substack{1\leq i_1,i_2,\cdots, i_k\leq n,\\ |\{i_1,\cdots,i_k\}|\leq k-1}} (a_{i_1i_2}+\frac{\mu}{\sqrt{n}})(a_{i_2i_3}+\frac{\mu}{\sqrt{n}})\cdots (a_{i_ki_1}+\frac{\mu}{\sqrt{n}})
.
\end{aligned}
\end{equation}
Expanding the product in equation \eqref{traceexpanecondjor}, a typical term in the expansion that is not in $\operatorname{Tr}(A_n)^k$ or $\operatorname{Tr}(\sqrt{n}C_n)^k$ is of the form $w_{\mathbf{i}}=w_{i_1i_2}\cdots w_{i_ki_1}$, where $r$ of the $w_{i_si_{s+1}}$ terms are some $\frac{\mu}{\sqrt{n}}$ and the other $k-r$ terms are some $a_{ij}$, for $r=1,2,\cdots,k-1$.

\begin{lemma}\label{fuckgreat}
    We have the following convergence in law:
\begin{equation}\label{eq2.14}
   \frac{\operatorname{Tr}((A_n+\sqrt{n}C_n)^k)-\operatorname{Tr}((A_n)^k)-\operatorname{Tr}((\sqrt{n}C_n)^k)}{n^{k/2}}\overset{\text{law}}{\underset{n\to\infty}\longrightarrow}0.
\end{equation}
\end{lemma}

\begin{proof} We first show that the expectation of the left hand side of \eqref{eq2.14} converges to zero. Recall that a typical term in the numerator of the LHS of \eqref{eq2.14} is of the form $w_{\mathbf{i}}=w_{i_1i_2}\cdots w_{i_ki_1}$, then in order for $\mathbb{E}[w_\mathbf{i}]\neq 0$, we need that each edge in $w_\mathbf{i}$ that are traversed by some $a_{ij}$ must be traversed twice (those travelled three times or more are negligible). Thus we can decompose the cycle $(i_1,i_2,\cdots,i_{k},i_1)$ into (a) a double cycle consisting of all the $a_{ij}$'s and some $\frac{\mu}{\sqrt{n}}$s, and (b) some other cycles attached to this cycle that consist entirely of $\frac{\mu}{\sqrt{n}}$ terms. Suppose the double cycles in (a) consist of $k_1$ vertices (and hence $k_1$ edges), and among the $k_1$ edges, $k_2$ of them are traversed by some $a_{ij}$'s and the other by $\frac{\mu}{\sqrt{n}}$.  Then there are remaining $k-2k_1$ edges to be added afterwards. This would account for an addition of strictly less than $k-2k_1$ new vertices. Thus the overall contribution to the expectation of the left hand side of \eqref{eq2.14} is, in the case $k>2k_1$: 
$$
n^{-\frac{k}{2}}\times O(n^{k_1}n^{k-2k_1-1}n^{-\frac{k-2k_1}{2}}n^{-(k_1-k_2)})=O(n^{-1}),
$$
and in the case $k=2k_1$, one must have $k_1>k_2$, so that the contribution is
$$
n^{-\frac{k}{2}}\times O(n^{k_1}n^{-(k_1-k_2)})=O(n^{-1}).
$$
This proves the expectation of the left hand side of \eqref{eq2.14} converges to $0$.

We then show the variance of the left hand side of \eqref{eq2.14} converges to $0$ as well. For two typical terms $w_{\mathbf{i}}=w_{i_1i_2}\cdots w_{i_ki_1}$ and $r_{\mathbf{j}}=r_{j_1j_2}\cdots r_{j_kj_1}$ to have non-zero covariance, we consider the graph formed by $\mathbf{i}\cup\mathbf{j}$, whose vertices and edges are the union of those from $\mathbf{i}$ and $\mathbf{j}$. Then each edge traversed by some $a_{ij}$ in $\mathbf{i}\cup\mathbf{j}$ should be traversed twice (those travelled more times are negligible), and such that there are equal number of terms in $w_\mathbf{i},r_\mathbf{j}$ travelled by $a_{ij}$'s, and they have the same subscript. Therefore $\mathbf{i}\cup\mathbf{j}$ is made up of (a) a double cycle with $k_1$ edges (each travelled twice) and $k_1$ vertices, of which $k_2$ are travelled by the $a_{ij}$'s; and (b) $2k-2k_1$ outlying edges, giving rise to at most $2k-2k_1-1$ new vertices. Then the overall contribution to the variance of the LHS of \eqref{eq2.14} is at most 
$$
n^{-k}\times O(n^{k_1}n^{2k-2k_1-1}n^{-\frac{2k-2k_1}{2}}n^{-(k_1-k_2)})=O(n^{-1}),
$$ in the case when $k<k_1$,
and 
$$
n^{-k}\times O(n^kn^{-(k_1-k_2)})=O(n^{-1}),$$
in the case when $k=k_1$ so that $k_1>k_2$,
completing the proof.
\end{proof}

\begin{proof}[\proofname\ of Theorem \ref{theorem1}, case (2)]
Thanks to Lemma \ref{fuckgreat}, and proving tightness of $q_n(z)$ just as the previous cases, we conclude that 
\begin{equation}
   q_n(z)\overset{\text{law}}{\underset{n\to\infty}\longrightarrow}(1-\mu z)\kappa(z)e^{-F(z)},\quad z\in\mathbb{D},
\end{equation}
regarded as convergence of holomorphic functions in $\operatorname{H}(\mathbb{D})$, so that when $|\mu|>1$, almost surely as $n$ gets large, there is precisely one eigenvalue of $n^{-1/2}A_n+C_n$ with modulus larger than $1+o(1)$, and this eigenvalue converges to $\mu$ as $n$ tends to infinity. 

More generally, if we assume that $|C_{i,j}|\leq \frac{M}{n}$ for some fixed $M>0$ and each $i,j=1,\cdots,n$, then most previous arguments still apply, and Lemma \ref{fuckgreat} is still valid. Thanks to the assumption that $c_n(z)$, the reverse characteristic function of $C_n$, converges to $c(z)$ in $\operatorname{H}(\mathbb{D})$, we deduce that
\begin{equation}
   q_n(z)\overset{\text{law}}{\underset{n\to\infty}\longrightarrow}c(z)\kappa(z)e^{-F(z)},\quad z\in\mathbb{D},
\end{equation}
so that by Rouche's theorem, eigenvalues with modulus larger than $1+o(1)$ converge to those $z\in\mathbb{C}$ that solves $c(z^{-1})=0$ and $|z|>1$, and vice versa.
This proves Theorem \ref{theorem1} under the case (2) stated there.
\end{proof}

\section{Perturbation of sparse i.i.d. matrices}
Now we are in the setting to prove Theorem \ref{theorem1.2}. We recall the trace asymptotic of the random matrix $A_n$ from \cite{coste2023sparse}:

\begin{theorem}[\cite{coste2023sparse}, Theorem 2.2 and Theorem 3.1] 
\begin{enumerate} The random matrix $A_n$ in the definition of Theorem \ref{theorem1.2} satisfies the following trace asymptotic:
\item In case (1) of Theorem \ref{theorem1.2}(sparse with mean degree $d$), let $Y_\ell,\ell\in\mathbb{N}^*$ a family of independent Poisson random variables with $Y_\ell\sim \operatorname{Poi}(d^\ell/\ell)$, and define 
\begin{equation}
    X_k:=\sum_{\ell\mid k}\ell Y_\ell,\quad (k\in\mathbb{N}^*).
\end{equation} Then we have the joint weak convergence 
\begin{equation}
    (\operatorname{Tr}(A_n^1),\cdots.\operatorname{Tr}(A_n^k))\overset{\text{law}}{\underset{n\to\infty}\longrightarrow}      (X_1,\cdots,X_k).
\end{equation}

\item In case (2) of Theorem \ref{theorem1.2}(semi-sparse with mean degree $d_n\to\infty,d_n=n^{o(1)}$), let $(N_k:k\geq 1)$ be a family of independent real Gaussian random variables with mean $0$, variance $1$. Then for any $k$, we have the joint convergence in law
\begin{equation}
    \left(\frac{\operatorname{Tr}(A_n)}{\sqrt{d_n}}-\sqrt{d_n},\cdots, \frac{\operatorname{Tr}(A_n^k)}{\sqrt{d_n}^k}-\sqrt{d_n}^k  \right)
    \overset{\text{law}}{\underset{n\to\infty}\longrightarrow}      (N_1,\cdots,\sqrt{k}N_k)+(0,1,0,1,\cdots ,1_{k\text{ even }}).
\end{equation}

\end{enumerate}    
\end{theorem}

In this sparsity regime we prove the following simple yet crucial lemma:

\begin{lemma}\label{lemma438} Assume the deterministic matrix $C_n$ satisfies the assumptions made in Theorem \ref{theorem1.2}. Then,
    \begin{enumerate}
        \item In case (1) of Theorem \ref{theorem1.2}, for each $k$ we have the following convergence in law
        \begin{equation}\label{fucg}
            \operatorname{Tr}((A_n+C_n)^k)-\operatorname{Tr}(A_n^k)-\operatorname{Tr}(C_n^k)    \overset{\text{law}}{\underset{n\to\infty}\longrightarrow} 0.
        \end{equation}
        \item In case (2) of Theorem \ref{theorem1.2}, for each $k$ we have the following convergence in law
        \begin{equation}\label{fucg2}
            \operatorname{Tr}((A_n+\sqrt{d_n}C_n)^k)-\operatorname{Tr}(A_n^k)-\operatorname{Tr}((\sqrt{d_n}C_n)^k)    \overset{\text{law}}{\underset{n\to\infty}\longrightarrow} 0.\end{equation}
    \end{enumerate}
\end{lemma}

\begin{proof}
    In both cases, expanding the trace of matrix power, the terms contributing to the left hand side of \eqref{fucg}, \eqref{fucg2} are of the form
    \begin{equation}\label{line452}
\sum_{\substack{1\leq i_1,\cdots,i_k\leq n,}} w_{i_1i_2}w_{i_2i_3}\cdots w_{i_ki_1},
    \end{equation}
where for some $s$, we have $w_{i_si_{s+1}}=a_{i_si_{s+1}}$ and for some $s'$ we have $w_{i_{s'}i_{s'+1}}=(\sqrt{d_n})C_{i_{s'},i_{s'+1}}$ (for case (1), we don't need the $\sqrt{d_n}$ term, but this is irrelevant to the forthcoming proofs). That is, both terms from $A_n$ and $C_n$ should be present in this product. The sequence $(i_1,i_2,\cdots,i_k,i_1)$ must form a cycle, and hence is a union of edge disjoint directed loops (including self loops). Suppose $t$ distinct indices are involved in the formation of this cycle, then there must be at least $t$ distinct edges involved in the cycle. Suppose that $m$ of these distinct edges are traversed by an element $(d_n)^{1/2}C_{ij}$, then since both vertices of such edges should come from indices where $C_{ij}\neq 0$ and there are only finitely many such vertices, we see that either (a) at least $m+1$ vertices are determined by these finitely many choices, or (b) these  edges traversed by $\sqrt{d_n}C_{j,k}$ should form a cycle. In case (a), we have only $O(n^{t-m-1})$ choices of free vertices. Meanwhile, for each $i,j$ we have $\mathbb{E}[a_{ij}]=\frac{d_n}{n}$, so the overall expectation of \eqref{line452} is $$O\left(n^{t-m-1}(\frac{d_n}{n})^{t-m}(d_n)^{m}\right)=O\left(\frac{(d_n)^{2k}}{n}\right).$$ Since by assumption $d_n=n^{o(1)}$, we see that contribution of such terms to the expectation of \eqref{line452} converges to zero. In case (b),we need to add additional cycles on the cycle already formed by the edges from $\sqrt{d_n}
C_{ij}$: we first add in those cycles involving edges from $\sqrt{d_n}C_{ij}$, and this case is exactly the same as in case (a), and what remains is cycles involving purely the $a_{ij}$ terms.   But we can see that an addition of $s$ vertices for such remaining cycles lead to at least $s+1$ more edges traversed by $a_{ij}$, and such terms contribute $O(\frac{(d_n)^{2k}}{n})$ to the expectation of \eqref{fucg},\eqref{fucg2} and is negligible as well. Combining both cases completes the proof.

\end{proof}

Now we can complete the proof of Theorem \ref{theorem1.2}.

\begin{proof}[\proofname\ of Theorem \ref{theorem1.2}]
We begin with case (1) of Theorem \ref{theorem1.2}.
    Denote by $$q_n(z)=\det(\operatorname{I}-z(A_n+C_n)).$$ Then modifying the arguments in \cite{coste2023sparse}, one can show that $\{q_n(z)\}_{n\geq 1}$ are tight as holomorphic functions on $\{z\in\mathbb{C}:|z|\leq\frac{1}{\sqrt{d}}\}$. By \cite{coste2023sparse}, Theorem 2.7 with a slight modification, combining the computations in Lemma \ref{lemma438}, we deduce that
    \begin{equation}
        q_n(z)\overset{\text{law}}{\underset{n\to\infty}{\longrightarrow}}e^{-\overline{f}(z)},
    \end{equation}
    where 
    $$
\overline{f}(z)=\sum_{k=1}^\infty (X_k+\operatorname{Tr}(A_n^k))\frac{z^k}{k},
    $$
    and this expression can be simplified as follows: denoting by
    $$
\mathscr{R}(z)=\sum_{k=1}^\infty (\tau_k-X_k)\frac{z^k}{k},\quad \tau_k=\mathbb{E}[X_k],
    $$
    we have 
    \begin{equation}
         q_n(z)\overset{\text{law}}{\underset{n\to\infty}{\longrightarrow}}e^{\mathscr{R}(z)}\times c(z)\times \prod_{\ell=1}^\infty (1-dz^\ell)^\frac{1}{\ell},
    \end{equation}
    and the convergence is in terms of convergence of holomorphic functions on $\{z\in\mathbb{C}:|z|<d^{-\frac{1}{2}}\}$. By \cite{coste2023sparse}, section 8, the factor $e^{\mathscr{R}(z)}$ is nonzero everywhere on $\{z\in\mathbb{C}:|z|<d^{-\frac{1}{2}}\}$, so by Rouche's theorem, almost surely for $n$ sufficiently large, the only eigenvalues of $A_n+C_n$ with modulus larger than $\sqrt{d}(1+o(1))$ converge either to $d$ or to those $z\in\mathbb{C}$ satisfying $c(z^{-1})=0$ and $|z|>\sqrt{d}$, and each root of the latter equation corresponds to the $n\to\infty$ limit of some outlier eigenvalues $\lambda_n$ of  $A_n+C_n$.

    Now we prove case (2) of Theorem \ref{theorem1.2}. For this it suffices to consider 
    $$
q_n(z)=\det(\operatorname{I}_n-z({d_n}^{-\frac{1}{2}}A_n+C_n)),
    $$
    Define 
    \begin{equation}
        G(z)=\exp\left\{\sum_{k=1}^\infty N_k\frac{z^k}{\sqrt{k}}\right\},
    \end{equation} then by modifying the proof of \cite{coste2023sparse}, Lemma 10.1, we can show $\frac{q_n(z)}{\sqrt{d_n}}$ is tight in $\operatorname{H}(\mathbb{D})$. Then modifying the proof of \cite{coste2023sparse}, Theorem 3.2, we have the convergence
    \begin{equation}
\frac{q_n(z)}{\sqrt{d_n}}\overset{\text{law}}{\underset{n\to\infty}{\longrightarrow}} -zc(z)\sqrt{1-z^2}G(z).
    \end{equation}
    in terms of convergence as functions on $\operatorname{H}(\mathbb{D})$. Therefore by Rouche's theorem, almost surely for $n$ sufficiently large, the eigenvalues of $(d_n)^{-1/2}A_n+C_n$ having modulus larger than $1+o(1)$ either converge to infinity (with multiplicity one), or converge to roots of $c(z^{-1})=0$ that satisfies $|z|>1$. Conversely, all the roots to $c(z^{-1})=0$ with $|z|>1$ arise as the $n\to\infty$ limit of some outlier eigenvalues $\lambda_n$ of $(d_n)^{-1/2}A_n+C_n$. This completes the proof.
\end{proof}

\section{Perturbation for the product of i.i.d. matrices}

This section covers the proof of Theorem \ref{productoftheorem1}. 
Recalling the definition of $\mathcal{Y}$ in \eqref{welinearizeit}, we define the reverse characteristic function via 
$$
q_n(z)=\det(1-z(n^{-1/2}\mathcal{Y}+\mathcal{C})),\quad z\in\mathbb{C}.
$$

\begin{proof}[\proofname\ of Theorem \ref{productoftheorem1}] We use
$\mathcal{Y}_{ij},\mathcal{C}_{ij}:1\leq i,j\leq mn$ to denote the entries of $\mathcal{Y}$ and $\mathcal{C}$.
Consider again the trace expansion
\begin{equation}\begin{aligned}
\operatorname{Tr}((\mathcal{Y}+\sqrt{n}\mathcal{C})^k)&=\sum_{\substack{1\leq i_1,i_2,\cdots,i_k\leq nm,\\|\{i_1,\cdots,i_k\}|=k}}(\mathcal{Y}_{i_1i_2}+\sqrt{n}\mathcal{C}_{i_1i_2})(\mathcal{Y}_{i_2i_3}+\sqrt{n}\mathcal{C}_{i_2i_3})\cdots (\mathcal{Y}_{i_ki_1}+\sqrt{n}\mathcal{C}_{i_ki_1})
\\&+\sum_{\substack{1\leq i_1,i_2,\cdots,i_k\leq nm,\\ |\{i_1,\cdots,i_k\}|\leq k-1}} (\mathcal{Y}_{i_1i_2}+\sqrt{n}\mathcal{C}_{i_1i_2})(\mathcal{Y}_{i_2i_3}+\sqrt{n}\mathcal{C}_{i_2i_3})\cdots (\mathcal{Y}_{i_ki_1}+\sqrt{n}\mathcal{C}_{i_ki_1}),\\
\end{aligned}
\end{equation}
we will follow the same lines of reasoning as in the proof of Theorem \ref{theorem1}, with the only difference here being that many entries $\mathcal{Y}_{i_si_{s+1}}$ are zero. As we are only showing the error converges to 0, this discrepancy is not important.

More specifically, we expand all the brackets on the right hand side, and denote by $T_1$ the terms consisting of monomials including both terms from $\mathcal{Y}_{ij}$ and $\mathcal{C}_{ij}$ and that $|\{i_1,\cdots,i_k\}|=k$, and let $T_2$ consist of the monomials including both terms from $\mathcal{Y}_{ij}$ and $\mathcal{C}_{ij}$ and that $|\{i_1,\cdots,i_k\}|\leq k-1$. We wish to prove that $\frac{T_1}{n^{k/2}}$ converges to $0$ in probability, and that $\frac{T_2}{n^{k/2}}$ converges to $0$ in probability.

To check the claim concerning $T_1$, we compute $\mathbb{E}[|T_1|^2]$. For two typical terms $w_{\mathbf{i}}=w_{i_1i_2}\cdots w_{i_ki_1}$ and $r_{\mathbf{j}}=r_{j_1j_2}\cdots r_{j_kj_1}$ (where each $w_{ij}$ and $r_{ij}$ are either some $\mathcal{Y}_{ij}$ or some $\sqrt{n}\mathcal{C}_{ij}$), in order for $\mathbb{E}[w_\mathbf{i}r_\mathbf{j}]\neq 0$, we require that each edge traversed by some (not identically zero) $\mathcal{Y}_{ij}$ must be traversed twice (those traversed more than twice has negligible contribution), and this forces there are exactly the same number of pairs in $w_\mathbf{i}$ and $r_\mathbf{j}$ that are some $\mathcal{Y}_{ij}$, and that they must correspond to the same $\mathcal{Y}_{ij}$ so as to appear at least twice. In this vein, assume that there are $Q$ edges in $w_\mathbf{i}$ and $r_\mathbf{j}$ traversed by some entries $\mathcal{Y}_{ij}$. We see that to reconstruct this graph we actually need no more than $Q-1$ free vertices chosen from any of $1,2,\cdots,nm$, as all other vertices are forced to be chosen from the subscript $i,j$ such that $\mathcal{C}_{ij}\neq 0$, which are finite by definition. Altogether, there are $2(k-Q)$ other edges traversed by some $\sqrt{n}\mathcal{C}_{ij}.$ Thus the overall contribution is $O(n^{Q-1}(\sqrt{n})^{2(k-Q)}),$ summing over these $\mathbf{i}\cup\mathbf{j}$. This readily implies $\frac{T_1}{n^{\frac{k}{2}}}\overset{\text{law}}{\underset{n\to\infty}\longrightarrow} 0.$

To check the claim concerning $T_2$, in order for a term $w_{i_1i_2}w_{i_2i_3}\cdots w_{i_ki_1}$ to have non-zero mean, each edge traversed by some (not identically zero) $\mathcal{Y}_{ij}$ must be traversed twice, so the graph formed by $(i_1,\cdots,i_k)$ can be decomposed as (a) a double cycle containing all the terms from $\mathcal{Y}_{ij}$ and some other terms from $\sqrt{n}\mathcal{C}_n$, and (b) some other cycles around it containing only elements from $\sqrt{n}\mathcal{C}_n$. Let $R$ be the number of edges from $\sqrt{n}\mathcal{C}_n$, then there are at most $\frac{k-R}{2}$ edges for the $\mathcal{Y}_{ij}$ terms (as each such edge should be traversed at least twice). This corresponds to $O(n^{\frac{k-R}{2}-1})$ choices of free vertices from $1,2,\cdots,nm$ and a bounded number of choices from vertices $i,j$ such that $\mathcal{C}_{i,j}\neq 0$. Here we have $O(n^{\frac{k-R}{2}-1})$ instead of $O(n^{\frac{k-R}{2}})$ choices because at least one vertices from the $\mathcal{Y}_{ij}$'s should come from the finitely many subscripts where $\mathcal{C}_{ij}\neq 0$. The overall contribution to $\mathbb{E}[T_2]$ is thus $O(n^{\frac{k-R}{2}-1}n^{\frac{R}{2}})$ and we have that $\mathbb{E}[T_2]/n^\frac{k}{2}$ vanishes in the limit. 

Finally we check $\operatorname{Var}(T_2)=o(n^{k})$ as follows: the computation is exactly the same as before. For two typical terms $w_{\mathbf{i}}=w_{i_1i_2}\cdots w_{i_ki_1}$ and $r_{\mathbf{j}}=r_{j_1j_2}\cdots r_{j_kj_1}$ to have non-zero covariance, each edge traversed by some $\mathcal{Y}_{ij}$ should be traversed twice, and assuming there are $Q$ such edges, this leads to $Q-1$ free vertices and $2(k-Q)$ other edges traversed by entries of $\sqrt{n}\mathcal{C}_n$. This implies $\operatorname{Var}(T_2)=O(n^{k-1})$. Combining both steps we conclude that $\frac{T_2}{n^{k/2}}\overset{\text{law}}{\underset{n\to\infty}\longrightarrow}0$.

To sum up, using also the fact that $c_n(z)$ converges to $c(z)$, and checking the tightness of $q_n(z)$ just as in Lemma \ref{tightnesslemma}, we have proved that 
\begin{equation}
    \lim_{n\to\infty }q_n(z)=\kappa_\eqref{productm2s}(z) c(z)e^{-F_{\ref{productm}}(z)}\quad z\in\mathbb{D},
\end{equation} in the sense of convergence as functions in $\operatorname{H}(\mathbb{D})$. This determines the limit of the outlying eigenvalues of $n^{-1/2}\mathcal{Y}+\mathcal{C}$. By a similar argument to Lemma \eqref{theorem2.1345}, we see that the outlying eigenvalues of $\prod_{i=1}^m(n^{-1/2}A^i+C^i)$ are the outlying eigenvalues of $n^{-1/2}\mathcal{Y}+\mathcal{C}$ raised to the $m$-th power, and thus the proof of Theorem \ref{productoftheorem1} is complete thanks to Rouche's theorem.
\end{proof}

\section*{Statements and declarations}
The author has no financial or non-financial conflicting of interests to declare that are relevant to the contents of this article.

No dataset is generated in this research.

\bibliographystyle{amsplain}

\end{document}